\documentclass[11pt]{article}
\usepackage[english]{babel}
\usepackage{amssymb, amsmath, amsthm} 
\usepackage[a4paper, centering]{geometry}
\usepackage{graphicx}
\usepackage{caption}
\usepackage{subfigure,color}
\usepackage[font={small,up}]{caption}

\numberwithin{equation}{section} 

\usepackage{pgfplots}
\usepackage{bm}

\usepackage{amsthm}
\makeatletter
\def\th@plain{%
  \thm@notefont{}
  \itshape 
}
\def\th@definition{%
  \thm@notefont{}
  \normalfont 
}
\makeatother
\newtheorem{thm}{Theorem}[section]
\newtheorem{prop}[thm]{Proposition}
\newtheorem{lem}[thm]{Lemma}
\newtheorem{cor}[thm]{Corollary}

\theoremstyle{definition}
\newtheorem{defn}[thm]{Definition}

\newtheorem{oss}[thm]{Remark}

\newcommand{\N}{\mathbb{N}}
\newcommand{\R}{\mathbb{R}}

\renewcommand{\epsilon}{\varepsilon}

\DeclareMathOperator{\dist}{dist}

\usepackage{enumitem}
\setlist[description]{nosep}

\title{A variational approach to the quasistatic limit of viscous dynamic evolutions in finite dimension}

\author{{\scshape Giovanni Scilla}\\
Department of Mathematics and Applications ``R. Caccioppoli''\\ University of Naples ``Federico II''\\
Via Cintia, Monte S. Angelo - 80126 Naples \\
(ITALY)
 \\
\\
{\scshape Francesco Solombrino}\\
Department of Mathematics and Applications ``R. Caccioppoli''\\ University of Naples ``Federico II''\\
Via Cintia, Monte S. Angelo - 80126 Naples \\
(ITALY)}

\date{}

\begin{document}

\maketitle

\begin{abstract}
In this paper we study the vanishing inertia and viscosity limit of a second order system set in an Euclidean space, driven by a possibly nonconvex  time-dependent potential satisfying very general assumptions. By means of a variational approach, we show that the solutions of the singularly perturbed problem converge to a curve of stationary points of the energy and characterize the behavior of the limit evolution at jump times. At those times, the left and right limits of the evolution are connected by a finite number of heteroclinic solutions to the unscaled equation.
\end{abstract}


\section{Introduction}
This paper is concerned with a variational analysis of the limit behavior of the system
\begin{equation}
\epsilon^2A\ddot{u}_\epsilon(t)+\epsilon B\dot{u}_\epsilon(t)+\nabla_{{x}} F(t,u_\epsilon(t))=0\,,
\label{mainequation}
\end{equation}
where $A$ and $B$ are positive definite and symmetric matrices and $F$ a time-depending driving potential, as the small parameter $\epsilon \to 0$. The above system describes the evolution of a mechanical system where both inertia and friction are taken into account, and can be considered as a second order approximation for the quasistatic evolution problem
\begin{equation}
\nabla_x F(t,u(t))=0,
\label{(1.2)}
\end{equation}
which appears in many areas of Applied Mathematics. In this respect, \eqref{mainequation} can be seen as a selection criterion for finding piecewise continuous solutions of \eqref{(1.2)}, since discontinuities are expected to appear if we allow $F$ for being nonconvex. As such, it has been used in several applications, even in an infinite-dimensional context.  We may mention, for instance,  \cite{DalSca}, where the solutions of the quasistatic evolution in linearly elastic perfect plasticity are approximated by the solutions of suitable dynamic visco-elasto-plastic problems; \cite{Sca}, where a ``vanishing viscosity and inertia'' limit is computed for a dynamic process in delamination; \cite{LN2,LN1},  where a ``vanishing inertia'' analysis is developed for a model of dynamic debonding
in the framework of fracture mechanics; \cite{Rubi,LRTT} for damage models, with a damping term in the wave equation. All these approaches build upon previous results in the simpler setting of vanishing viscosity (see, e.g., \cite{Mielke, EM, MRS1, MRS2, MT, DalSca, Ago1} and the references therein). There, problem \eqref{(1.2)} is seen as the limiting case of a system governed by an overdamped dynamics, that is 
\begin{equation}
\epsilon \dot{u}_\epsilon(t)+\nabla_{{x}} F(t,u_\epsilon(t))=0.
\label{problem1}
\end{equation}

In this paper we aim at providing  a general variational approach to the limit description of \eqref{mainequation}, extending some recent results for the first-order system \eqref{problem1}. For the moment, we will confine ourselves to a {\it finite-dimensional} setting and to a {\it smooth} driving potential $F$. We have to warn the reader that many of the mentioned applications, instead, deal with infinite dimensional rate-independent evolutions. In this case (see \cite[Introduction and Section 7]{Negri1}), $F$ usually takes the form
\[
F(t,u)=\underbrace{E(t,u)}_{\mbox{stored energy}}+\underbrace{D(u)}_{\mbox{dissipation potential}}\,,
\]
where the existence of $D$ involves some constraints on the admissible increments, or equivalently forces to allow for nonsmoothness. However, many nontrivial issues already arise in our simpler setting, where significant steps towards a general understanding of the problem can be made. Before describing our approach in detail, we recall some recent abstract results on the limit behavior of systems of the type \eqref{mainequation} or \eqref{problem1}.

\noindent
{\bf Results present in literature.} The asymptotic behavior of the solutions of singularly perturbed differential systems in finite dimension has been investigated by several authors \cite{Ago1, Ago-Rossi, N, Zanini, Sci-Sol}. A first general approach to the limit behavior of the solutions of \eqref{mainequation} has been developed by Agostiniani~\cite{Ago1}, extending a previous approach by Zanini \cite{Zanini} for the vanishing viscosity problem \eqref{problem1}. Under suitable assumptions on $F(t,x)$, it is proven that, when $\epsilon \to 0$, it holds
\begin{equation*}
(u_\epsilon(t),\epsilon B\dot{u}_\epsilon(t))\to (u(t),0),
\end{equation*}
where $u$ is a piecewise-continuous function solving
\begin{equation}
\nabla_{{x}} F(t,u(t))=0
\label{stazintro}
\end{equation}
at every continuity time $t$. Moreover, the trajectories of the system at the jump times $t_i$ are described through the autonomous second order system
\begin{equation}
A\ddot{w}(s)+B\dot{w}(s)+\nabla_{{x}} F(t_i, w(s))=0,
\label{het}
\end{equation}
with conditions $\displaystyle\lim_{s\to-\infty}w(s)=u_-(t_i)$, $\displaystyle\lim_{s\to+\infty}w(s)=u_+(t_i)$ where $\nabla_{{x}} F(t_i, u_+(t_i))=0$,  and $\displaystyle\lim_{s\to\pm\infty}\dot{w}(s)=0$. 

It is worth noting that the presence of the damping term $\epsilon B\dot{u}_\epsilon$ is crucial for obtaining the above results, as it also will be in our setting. There are indeed examples of singularly perturbed second order potential-type equations (with vanishing inertial term), such that the dynamic solutions do not converge to equilibria, while the formal limit equation is \eqref{(1.2)} (see, e.g.,~\cite{N}). 

As a matter of fact, the set of assumptions considered in \cite{Ago1} involves some significant restrictions. First of all,  a central role in the constructive approach contained there is played by the so called \emph{ transversality conditions} (see \cite[Assumption 2]{Zanini}), holding at degenerate critical points of $F(t,\cdot)$. Although the genericity of such assumption, it excludes some interesting situations, like bifurcation from a trivial critical state with change of stability. Even more cumbersome is the fact that one has to assume that the limit points $u_+(t_i)$ of the heteroclinic trajectories governed by \eqref{het} are strict local minimizers of $F(t_i, \cdot)$ (\cite[Assumption 4]{Ago1}), while in general they could even be saddle points.

Therefore, we take a different viewpoint  of variational character. Our starting point is the paper by Agostiniani and Rossi \cite{Ago-Rossi}, concerning the limit behavior of the first-order system \eqref{problem1}.  Along with  typical regularity, coercivity and power control assumptions on $F$ (which we also consider, see {\bf (F0)-(F2)}), a crucial role in the analysis is played by the assumption that the set of critical points $\mathcal{C}(t):=\left\{u\in X:\, \nabla_x F(t,u)=0\right\}$ consists of {\it isolated points} (also this one is assumed in  our paper, see {\bf(F5)} below). This allows for recovering the necessary  compactness through a careful analysis of the behavior at jumps. They indeed show that, up to a subsequence, the solutions $u_\epsilon$ of \eqref{problem1} pointwise converge, as $\epsilon\to0$, to a so-called \emph{Balanced Viscosity solution} $u$ of the limit problem \eqref{stazintro} defined at every $t\in[0,T]$. The function $u$ is \emph{regulated}, i.e., the left and right limits $u_-(t)$ and $u_+(t)$ exist at every $t$, and satisfies the stability condition
\begin{equation*}
\nabla_xF(t,u_-(t))=\nabla_xF(t,u_+(t))=0.
\end{equation*}
Furthermore,  under additional assumptions (in the same spirit of our assumptions {\bf (F6)-(F3')}), $u$ fulfills the \emph{energy balance} 
\begin{equation*}
F(t,u_+(t))+\mu([s,t])=F(s,u_-(s))+\int_s^t \partial_r F(r,u(r))\,dr\quad \forall\,0\leq s\leq t\leq T,
\end{equation*} 
where $\mu$ is a positive pure jump measure with an at most countable support $J$, that coincides with the jump set of $u$. For $t\in J$, the following jump relation hold:
\begin{equation*}
\mu(\{t\})=F(t,u_-(t))-F(t,u_+(t))=c_t(u_+(t),u_-(t))\,,
\end{equation*}
where the cost $c_t$ is defined as
\begin{equation}
c_t(u_1,u_2):=\inf\left\{\int_0^1\|\dot{v}(s)\|\|\nabla_x F(t, v(s))\|\,ds:\,v(0)=u_1, v(1)=u_2\right\}\,.
\label{costfunintro}
\end{equation}
In particular, at a jump point $t\in J$, transitions between (meta)stable states of the energy happen along (a finite union of) heteroclinic orbits of the unscaled autonomous gradient flow
\[
\dot w(s)=-\nabla_x F(t, w(s))\,.
\]
\noindent
{\bf Description of our results.} We now turn to the  description of our results. A first step is to investigate the compactness properties of the sequence $(u_\epsilon(t))_\epsilon$. 
To this end, we preliminarly provide some \emph{a-priori} estimates (Proposition~\ref{apriori}), namely $L^\infty$-bounds on $u_\epsilon$, $\epsilon\dot{u}_\epsilon$ and $\epsilon^2\ddot{u}_\epsilon$, and bounds for $\epsilon\|\dot{u}_\epsilon\|^2_{L^2}$, $\frac{1}{\epsilon}\|\nabla_x F\|^2_{L^2}$ and $\epsilon^3\|\ddot{u}_\epsilon\|^2_{L^2}$. 
These bounds stem out of \eqref{mainequation} and the energy equality
\begin{equation}
\begin{split}
\frac{\epsilon^2}{2} \|\dot{u}_\epsilon(t)\|^2+F(t,u_\epsilon(t))+\epsilon\int_s^t \|\dot{u}_\epsilon(\tau)\|^2\,\mathrm{d}\tau=\\
\frac{\epsilon^2}{2} \|\dot{u}_\epsilon(s)\|^2+F(s,u_\epsilon(s))+\int_s^t\partial_\tau F(\tau,u_\epsilon(\tau))\,\mathrm{d}\tau,
\label{inequi}
\end{split}
\end{equation}
(we assume here for simplicity of exposition that the matrices $A$ and $B$ in \eqref{mainequation} are equal to the identity matrix), but some care has to be used to estimate  $\frac{1}{\epsilon}\|\nabla_x F\|^2_{L^2}$ and $\epsilon^3\|\ddot{u}_\epsilon\|^2_{L^2}$ separately. In particular, we are forced to require more regularity on the energy $F(t,x)$ with respect to the first-order case analysed in \cite{Ago-Rossi}. Namely, with assumption {\bf (F4)}, we consider $F(t,\cdot)\in C^2(X)$ for every $t\in[0,T]$ and the Hessian matrix $\nabla_x^2 F(t,u)$ to be continuous in the product space $[0,T]\times X$. We remark nevertheless that the same assumption was already present in \cite{Ago1}.

With the aforementioned a-priori bounds, under assumptions {\bf (F0)-(F5)} (see Section~\ref{assumptions}), one can pass to the limit along a subsequence independent of $t$ on $u_\epsilon(t)$ (Theorem ~\ref{compact}), and find that they converge for all $t$ to a regulated function $u(t)$, whose jump set $J$ is at most countable. This limit function in general satisfies the stability condition
\begin{equation}
\nabla_x F(t,u_-(t))=\nabla_x F(t,u_+(t))=0
\label{stazintro2}
\end{equation}
for all $t \in [0,T]$. 

The proof moves from the remark that, by the previous a priori estimates, the sequence of positive measures
\begin{equation*}
\mu_\epsilon:= \epsilon\|\dot{u}_\epsilon(\cdot)\|^2\mathcal{L}^1
\end{equation*} 
is equibounded in $L^1(0,T)$ and then is weakly* converging to a positive finite measure $\mu$ on $[0,T]$, whose set of atoms $J$ is at most countable. The key point is showing that oscillations in the limit of the sequence $u_\epsilon(t)$ always happen at a nonvanishing cost, and therefore a limit is uniquely determined for each $t\notin J$. To this end, a crucial role is played by Proposition~\ref{propdelta}, where, exploiting {\bf (F5)} and assuming that $s_k$, $t_k$ are sequences in $[0,T]$ both converging to $t$ and such that $u_{\epsilon_k}(s_k)\to u_1$, $u_{\epsilon_k}(t_k)\to u_2$, with $u_1\neq u_2$, it is shown that the dissipation integrals 
\begin{equation}
\int_{s_k}^{t_k}\|\nabla_xF(r,u_{\epsilon_k}(r))\|\|\dot u_{\epsilon_k}(r)\|\mathrm{d}r
\end{equation}
are bounded away from zero for $k$ large enough. The resulting continuous limit function $u(t)$ complies with the stability condition \eqref{stazintro2} at every $t\in[0,T]\backslash J$. The existence of left and right limits of $u(t)$ still relies on Proposition~\ref{propdelta} and on the asymptotic and monotonicity properties of the functions
\begin{equation*}
g_\epsilon(t):= F(t,u_\epsilon(t))+\frac{\epsilon^2}{2}\|\dot{u}_\epsilon(t)\|^2-\int_0^t \partial_rF(r,u_\epsilon(r))\,\mathrm{d}r
\end{equation*}
via Helly's Theorem.

After that compactness and stability properties of the limit evolution have been established, we show that the limit evolution $u(t)$ satisfies a balance between the stored energy and  the power spent  along the evolution in an interval of time $[s, t]\subset [0, T]$, up to a positive dissipation cost which is concentrated on the jump set of $u$, or equivalently on the jump set of the energy $t\longmapsto F(t, u(t))$. Namely, we prove in Theorem~\ref{genbalance} that there exists a {\it positive atomic measure} $\mu$, with ${\rm supp}(\mu)=J$, such that
\begin{equation}
F(t,u_+(t))+\mu([s,t])=F(s,u_-(s))+\int_s^t \partial_r F(r,u(r))\,dr,
\label{stabilityidentity}
\end{equation} 
for all $0\leq s\leq t\leq T$. 
While the ``$\,\le\,$'' inequality in \eqref{stabilityidentity} can be obtained passing to the limit in the energy inequality deriving from \eqref{inequi}, the ``$\,\ge\,$'' ensues from the stability condition, and requires the additional assumptions {\bf(F6)-(F3')} on the energy, which are instead not necessary in order to recover compactness (see Section \ref{assumptions}).

We retrieve for $u$ an analogous of the notion of {\it Balanced Viscosity solution} introduced in the first-order setting in \cite{Ago-Rossi}, as we show (Remark~\ref{equality}) that, for all $t \in J$,
\begin{equation}
F(t,u_-(t))-F(t,u_+(t))=\mu(\{t\})=c_t(u_+(t),u_-(t))\,.
\label{jumps}
\end{equation}
The cost $c_t$ is, as one may expect, different from the cost (\ref{costfunintro}) considered in \cite{Ago-Rossi},  since it reflects the second-order structure of our problem. Assuming for simplicity of exposition that the matrices $A$ and $B$ in \eqref{mainequation} are equal to the identity matrix, it is actually defined as
\begin{equation}
c_t(u_1,u_2):=\displaystyle\inf\left\{\frac{1}{2}\int_{-N}^{N}\|\nabla_xF_t(v(s))+\ddot{v}(s)\|^2+\|\dot{v}(s)\|^2\,\mathrm{d}s:\,v\in V^{t, N}_{u_1,u_2},\,N\in \mathbb{N}\right\},
\label{costintro}
\end{equation}
where 
\begin{equation*}
V^{t,N}_{u_1,u_2}:=\Bigl\{v\in W^{2,2}([-N, N],X):\, v(-N)=u_1, v(N)=u_2, \dot v(-N)=\dot v(N)=0\Bigr\},
\end{equation*}
and $u_1,u_2\in X$. 

Exploiting {\bf (F5)} and some general properties of the cost $c_t$, proved in Theorem~\ref{cost}, by means of an inductive construction, we are indeed able to show that an optimal decrease of the energy can be realised at discontinuities via a {\it finite} number of transitions between metastable states. This gives the desired equivalence and entails our first main result, Theorem \ref{main1}. There we show that the limit $u(t)$ is exactly a Balanced Viscosity solution in the sense discussed above.

From the jump conditions \eqref{jumps} we can deduce a variational description of the behavior of the limit evolutions at jumps. More precisely, in  Theorem~\ref{variat} we prove that, if \eqref{jumps} holds, then every infimizing sequence for $c_t(u_-(t),u_+(t))$ converges to a finite union of heteroclinic solutions to the unscaled problem
\begin{equation}
A\ddot{v}^i(s)+B\dot{v}^i(s)+\nabla_xF_t(v^i(s))=0,\quad\forall\,s\in\R.
\label{unscaled}
\end{equation}
Each $v^i$ connects two distinct critical points $u_{i-1}$, $u_i$ of $\nabla_x F_t(\cdot)$ in the sense that $v^i(-\infty)=u_{i-1}$, $v^i(+\infty)=u_i$, and the endpoints of this finite chain of critical points are exactly $u_-(t)$ and $u_+(t)$. In this way, the results of \cite{Ago1} are extended to a general framework of driving potentials. We also believe that the approach pursued here can represent a solid building block for the understanding of related infinite dimensional problems.
\\
{\bf Plan of the paper.} The paper is organized as follows. In Section~\ref{prel} we fix notation and introduce the main assumptions on the energies $F(t,x)$ we will adopt throughout the paper. Section~\ref{preliminary} contains some preliminary results useful in the sequel, as basic inequalities, a priori estimates and other technical tools.
In Section~\ref{compactness}, we prove  compactness of the $u_\epsilon$ and stability properties of the limit evolution.
In Section~\ref{energybalance} we show that the limit evolution $u$ fulfills an energy balance with a cost concentrated on the jump points (Theorem~\ref{genbalance}). Finally, with Theorem \ref{variat} we provide a variational characterization of the behavior of the limit evolution at the jump times, by showing that the left and right limits $u_-(t)$ and $u_+(t)$, respectively, are connected by a finite number of heteroclinic solutions of the unscaled autonomous equation \eqref{unscaled}.

\section{Notation and main assumptions on the energy}\label{prel}

In this section, we fix some general notation that will be used throughout. We aim at describing second order quasistatic evolutions driven by  a time-dependent, possibly nonconvex energy functional $F: [0, T]\times X\longrightarrow\R$, with $T>0$.
Throughout the paper we assume that $(X,\|\cdot\|)$ is a Euclidean space with dimension $n\geq1$, endowed with inner product $\langle \cdot, \cdot \rangle$.  For a symmetric, positive definite operator  $Q:X\longrightarrow X$ the equivalent norm $\|\cdot\|_Q$ on $X$ is given by $\|u\|_Q:= \langle u, Qu\rangle^{\frac12}$. It holds  $\|u\|_Q=\|Q^{\frac12}u\|$, where $Q^{\frac12}$ is the principal square root of $Q$. With this, the Cauchy-Schwarz inequality takes the form
\begin{equation}
|\langle z_1,z_2\rangle|\leq \|z_1\|_{Q^{-1}}\|z_2\|_{Q}\leq \frac{1}{2}(\|z_1\|_{Q^{-1}}^2+\|z_2\|_{Q}^2),
\label{CSequi}
\end{equation}
for every $z_1,z_2\in X$. We also point out the simple identity
\begin{equation}
\|z_1+Qz_2\|_{Q^{-1}}^2=\|z_1\|_{Q^{-1}}^2+2\langle z_1,z_2\rangle+\|z_2\|_Q^2
\label{simpleid}
\end{equation}
for every $z_1,z_2\in X$. 
The symbol $\|Q\|$ will stand for the operator norm of $Q$.

Given $x \in X$ and $M > 0$, we will denote by $B_M(x)$ the closed ball centered at $x$ with radius $M$. When the ball is centered at $0$, the shortcut $B_M$  will be used. In order to shorten notation, we will often indicate by $u(-\infty)$ and $u(+\infty)$ the limits
\begin{equation*}
\displaystyle \lim_{s\to-\infty}u(s),\quad \displaystyle \lim_{s\to+\infty}u(s),
\end{equation*} 
respectively.

We now recall the definition and some basic properties of regulated functions, which will play an important role in the sequel.

\begin{defn}
A function $u:[0,T]\longrightarrow X$ is said to be \emph{regulated} if for each $s\in[0,T]$ there exist the one-sided limits $u_+(s)$ and $u_-(s)$ defined as
\begin{equation*}
u_+(s):=\lim_{h\to0^+} u(s+h),
\end{equation*}
and
\begin{equation*}
u_-(s):=\lim_{h\to0^-} u(s+h).
\end{equation*}
\end{defn}

The existence of the above limits immediately implies that, for each $N \in \mathbb{N}$, the set of points $t$ where $\|u_+(t)-u_-(t)\|\ge \frac1N$ cannot have accumulation points. It follows that the jump set of  a regulated function is at most countable. In particular, $u_+$ is a right-continuous Lebesgue representative of $u$ and $u_-$ is a left-continuous one. 

It is well-known that a function of bounded variation  $f\in {\rm BV}([a,b];\R)$ is a real-valued regulated function. The representatives $f_+$ and $f_-$ are in this case good representatives in the sense of \cite[Theorem 3.28]{AFP}: as shown there, the distributional derivative $Df$ (which is a Radon measure) satisfies 
\begin{equation}\label{teoremafond}
Df([s,t])=f_+(t)-f_-(s)
\end{equation}
for any $s,t\in[a,b]$ with $s\le t$.

\subsection{Assumptions on the energy}\label{assumptions}

We require the energy functional satisfy the following assumptions:\\

\begin{description}
\item[(F0)] $F\in C^1([0,T]\times X)$;\\
\item[(F1)] the map $\mathcal{F}:u\longmapsto \displaystyle\sup_{t\in[0, T]}|F(t,u)|$ satisfies the condition that, for every $\rho>0$, the sublevel set $\{u\in X:\, \mathcal{F}(u)\leq \rho\}$ is bounded;\\

\item[(F2)] there exist $C_1,C_2>0$ such that 
\begin{equation}
 |\partial_t F(t,u)|\leq C_1 F(t,u)+C_2,\quad \forall (t,u)\in[0,T]\times X,
\end{equation}
where $\partial_t F$ denotes the partial derivative of $F(t,x)$ with respect to $t$;\\
\item[(F3)] for every $t\in[0,T]$ and $u\in B_M$, it holds
\begin{equation}
|\partial_t\nabla_x F(t,u)|\leq a_M(t),
\label{estim}
\end{equation}
for some function $a_M\in L^1(0,T)$.\\
\item[(F4)] $F(t,\cdot)\in C^2(X)$ for every $t\in[0,T]$ and the Hessian matrix $\nabla_x^2 F(t,u)$ is continuous on $[0,T]\times X$;\\
\item[(F5)] for any $t\in[0,T]$, the set of critical points
\begin{equation}
\mathcal{C}(t):=\{u\in X:\, \nabla_x F(t,u)=0\},
\label{criticalset}
\end{equation}
where $\nabla_x F$ denotes the gradient of $F$ with respect to $x$, consists of isolated points.
\end{description}

We note that from {\bf (F2)} and the Gronwall's inequality we get
\begin{equation*}
F(t,u)\leq \left(F(s,u)+C_2\right)e^{C_1(t-s)}-C_2,
\end{equation*}
for all $0\le s\le t\le T$ and $u\in X$. This implies, in particular, that
\begin{equation}
\sup_{t\in[0,T]}|F(t,u)|\leq (F(0,u)+C_2)e^{C_1T}-C_2,\quad\text{ for all }u\in X,
\label{equibound}
\end{equation}
which will be useful to deduce equi-boundedness estimates.

The above assumptions {\bf (F0)-(F5)} will be enough to establish compactness of the limit functions in Section~\ref{compactness}, while in Section~\ref{energybalance} we will be forced to consider an additional assumption {\bf(F6)} and to strengthen assumption {\bf(F3)} in order to recover an energy balance.
We will namely assume that the driving energy $F(t,x)$ satisfies:

\begin{enumerate} 
\item[{\bf (F6)}] For any $t\in[0,T]$ and for any $u\in \mathcal{C}(t)$, 
\begin{equation*}
\displaystyle \mathop{\lim\inf}_{v\to u}\frac{F(t,v)-F(t,u)}{\|\nabla_xF(t,v)\|}\geq0.
\end{equation*}
\item[{\bf (F3')}] For  fixed $u \in X$, the function $t\longmapsto \nabla_xF(t,u)$ is Lipschitz continuous on $[0,T]$, locally uniformly w.r.t. $u$. 
\end{enumerate}
Condition {\bf(F3')} is satisfied, e.g., if the mixed derivative $\partial_t\nabla_xF(\cdot,\cdot)$ is continuous on $[0,T]\times X$. Thus, it is a reinforcement of assumption {\bf (F3)}, since it amounts to require that $a_M$ therein be bounded on $[0,T]$.

The generic character of assumptions {\bf(F5)-(F6)}, which are satisfied by a broad class of potentials, is discussed in detail in \cite[Section~2.1]{Sci-Sol}.
It is shown there that {\bf (F6)} holds, for instance, if $F(t, \cdot)$ is convex for fixed $t$, but also whenever $F(t,\cdot)$ complies with the \L ojasiewicz inequality.

\section{Preliminary results}\label{preliminary}

We state and prove here some preliminary results. In particular, we show that each solution $u_\epsilon$ to \eqref{mainequation} complies, for every fixed $\epsilon$ and for every $t\in[0,T]$, with an energy identity that will be a useful tool for the sequel.

\begin{lem}
Under assumptions {\bf (F0)-(F2)} there exists a unique $u_\epsilon:[0,T]\longrightarrow X$ of class $C^2$, solution of the Cauchy problem associated to \eqref{mainequation} with initial condition at $t=0$. Moreover, the following \emph{energy identity} holds:
\begin{equation}
\begin{split}
\epsilon\int_s^t\|\dot{u}_\epsilon(\tau)\|_B^2\,\mathrm{d}\tau+\frac{\epsilon^2}{2} \|\dot{u}_\epsilon(t)\|_A^2&+F(t,u_\epsilon(t))\\
&=\frac{\epsilon^2}{2} \|\dot{u}_\epsilon(s)\|_A^2+F(s,u_\epsilon(s))+\int_s^t\partial_\tau F(\tau,u_\epsilon(\tau))\,\mathrm{d}\tau,
\end{split}
\label{zerostima}
\end{equation}
for every $0\leq s<t\leq T$.
\end{lem}

\proof
Since the power control {\bf (F2)} provides, in particular, the boundedness from below of the energy $F$, the (local) existence of $u_\epsilon$ is ensured by a standard argument in ODEs for Cauchy problems associated to \eqref{mainequation}. Testing \eqref{mainequation} by $\dot{u}_\epsilon$ we get 
\begin{equation}
\begin{split}
&\epsilon^2\langle A\ddot{u}_\epsilon(r), \dot{u}_\epsilon(r)\rangle + \epsilon\langle B\dot{u}_\epsilon(r), \dot{u}_\epsilon(r) \rangle+\langle\nabla_xF(r,u_\epsilon(r)),\dot{u}_\epsilon(r)\rangle \\
&=\frac{\epsilon^2}{2}\frac{\mathrm{d}}{\mathrm{d}r}\langle A\dot{u}_\epsilon(r), \dot{u}_\epsilon(r) \rangle + \epsilon \langle B\dot{u}_\epsilon(r), \dot{u}_\epsilon(r) \rangle + \frac{\mathrm{d}}{\mathrm{d}r}F(r, u_\epsilon(r))-\partial_r F(r,u_\epsilon(r))=0,
\end{split}
\end{equation}
whence, integrating in time between $s$ and $t$, with $s,t\in[0,T]$ we obtain
\begin{equation*}
\begin{split}
\frac{\epsilon^2}{2} \|\dot{u}_\epsilon(t)\|_A^2-\frac{\epsilon^2}{2} \|\dot{u}_\epsilon(s)\|_A^2+\epsilon\int_s^t\|\dot{u}_\epsilon(\tau)\|_B^2\,\mathrm{d}\tau&+F(t,u_\epsilon(t))-F(s,u_\epsilon(s))\\
&-\int_s^t\partial_\tau F(\tau,u_\epsilon(\tau))\,\mathrm{d}\tau=0,
\end{split}
\end{equation*}
that corresponds to \eqref{zerostima}.

\endproof

As an easy remark, we note that just neglecting the non-negative term $\displaystyle\epsilon\int_s^t\|\dot{u}_\epsilon(\tau)\|_B^2\,\mathrm{d}\tau$ in \eqref{zerostima}, we get the inequality
\begin{equation}
\frac{\epsilon^2}{2} \|\dot{u}_\epsilon(t)\|_A^2+F(t,u_\epsilon(t))\leq\frac{\epsilon^2}{2} \|\dot{u}_\epsilon(s)\|_A^2+F(s,u_\epsilon(s))+\int_s^t\partial_\tau F(\tau,u_\epsilon(\tau))\,\mathrm{d}\tau,
\label{unostima}
\end{equation}
for every $0\leq s<t\leq T$. We will often refer to \eqref{unostima} as the \emph{energy inequality}.\\     \\

The following proposition collects some \emph{a priori} bounds, involving $u_\epsilon$ and its derivatives, that will provide key estimates in the proofs of the main results of this paper.

\begin{prop}[A priori bounds]\label{apriori}
Assume $\bf (F0)-(F4)$. Let $u_\epsilon:[0,T]\longrightarrow X$ of class $C^2$ be the solution of the Cauchy problem associated to \eqref{mainequation} with initial condition at $t=0$, and assume $u_\epsilon(0)$, $\epsilon\dot{u}_\epsilon(0)$ to be bounded.  Then for all $\epsilon>0$ and $t\in[0,T]$, the following a priori bounds hold:
\begin{enumerate}
\item[\rm (i)] $\|u_\epsilon(t)\|\leq C$;
\item[\rm (ii)] $\epsilon\|\dot{u}_\epsilon(t)\|\leq C$;
\item[\rm (iii)] $\epsilon^2\|\ddot{u}_\epsilon(t)\|\leq C$;
\item[\rm (iv)] $\displaystyle\epsilon\int_0^T\|\dot{u}_\epsilon(r)\|^2\,\mathrm{d}r\leq C$;
\item[\rm (v)] $\displaystyle\frac{1}{2\epsilon}\int_0^T\|\nabla_xF(r,u_\epsilon(r))+\epsilon^2A\ddot{u}_\epsilon(r)\|^2\,\mathrm{d}r\leq C$;
\item[\rm (vi)] $\displaystyle\epsilon\left|\int_0^T\langle \ddot{u}_\epsilon(r), \nabla_x F(r,u_\epsilon(r))\rangle\,\mathrm{d}r\right|\leq C$;
\item[\rm (vii)] $\displaystyle\frac{1}{2\epsilon}\int_0^T\|\nabla_xF(r,u_\epsilon(r))\|^2\,\mathrm{d}r\leq C$;
\item[\rm (viii)] $\displaystyle{\epsilon^3}\int_0^T\|\ddot{u}_\epsilon(r)\|^2\,\mathrm{d}r\leq C$.
\end{enumerate}
\end{prop}

\proof
(i) As a consequence of \eqref{equibound} we deduce the equi-boundedness of $F(t,u_\epsilon(t))$ and, in view of {\bf (F1)}, the compactness of $(u_\epsilon(t))_\epsilon$ at every $t$. In particular, we obtain (i).\\
(ii) By \eqref{unostima} it holds
\begin{equation*}
\epsilon\|\dot{u}_\epsilon(t)\|_A\leq \sqrt{2\left(|F(t,u_\epsilon(t))|+|F(0,u_\epsilon(0)|\right)+\epsilon^2\|\dot{u}_\epsilon(0)\|_A^2+2C_1\int_0^TF(r,u_\epsilon(r))\,\mathrm{d}r+2C_2T}.
\end{equation*}
Since $\|\cdot\|_A$ is an equivalent norm, the estimate follows from {\bf (F0)-(F2)}, \eqref{equibound}, (i) and the boundedness of $\dot{u}_\epsilon(0)$, $\epsilon\dot{u}_\epsilon(0)$. 
\\
(iii) From \eqref{mainequation} we get
\begin{equation*}
\epsilon^2\|A\ddot{u}_\epsilon(t)\|\leq \epsilon\|B\dot{u}_\epsilon(t)\|+\|\nabla_xF(t,u_\epsilon(t))\|\,.
\end{equation*}
Since $A$ is coercive, (iii) follows as a consequence of (i), (ii) and {\bf (F0)}.\\
(iv) By arguing as for (ii), the bound (iv) is a consequence of \eqref{zerostima}, (i), (ii), {\bf (F0)-(F2)}, \eqref{equibound} and the boundedness of $\dot{u}_\epsilon(0)$, $\epsilon\dot{u}_\epsilon(0)$.\\
(v) From \eqref{mainequation} we deduce the identity
\begin{equation*}
\frac{\epsilon}{2}\int_0^T \|B\dot{u}_\epsilon(r)\|^2\,\mathrm{d}r=\frac{1}{2\epsilon}\int_0^T\|\nabla_xF(r,u_\epsilon(r))+\epsilon^2A\ddot{u}_\epsilon(r)\|^2\,\mathrm{d}r,
\end{equation*}
from which the desired bound follows by applying (iv).\\
(vi) An integration by parts gives
\begin{equation*}
\begin{split}
&\epsilon\int_0^T\langle A \ddot{u}_\epsilon(r), \nabla_x F(r,u_\epsilon(r))\rangle\,\mathrm{d}r\\
&=\epsilon \Big[\langle \nabla_x F(\tau,u_\epsilon(\tau)),A\dot{u}_\epsilon(\tau) \rangle \Big|_0^T- \int_0^T\langle \partial_r\nabla_xF(r,u_\epsilon(r))+\nabla_x^2F(r,u_\epsilon(r))\dot{u}_\epsilon(r),A\dot{u}_\epsilon(r)\rangle\,\mathrm{d}r\Big];
\end{split}
\end{equation*}
moreover, by {\bf (F3)} and (ii) we get
\begin{equation*}
\begin{split}
\epsilon\int_0^T|\langle \partial_r\nabla_xF(r,u_\epsilon(r)), A\dot{u}_\epsilon(r)\rangle|\,\mathrm{d}r&\leq\epsilon\|A\|\int_0^T|\partial_r\nabla_xF(r,u_\epsilon(r))| \|\dot{u}_\epsilon(r)\|\,\mathrm{d}r\\
&\leq C\|a_C\|_{L^1(0,T)}
\end{split}
\end{equation*}
while by {\bf (F4)} we obtain
\begin{equation*}
\epsilon\int_0^T|\langle \nabla_x^2F(r,u_\epsilon(r))\dot{u}_\epsilon(r), \dot{u}_\epsilon(r)\rangle|\,\mathrm{d}r\leq \tilde{C}\|A\|\epsilon\int_0^T\|\dot{u}_\epsilon(r)\|^2\,\mathrm{d}r,
\end{equation*}
where $\tilde{C}$ is a uniform bound for $ \nabla_x^2F(t,u)$ on the compact set $[0,T]\times B_C$. Thus, combining these estimates with (iv), the boundedness of $\nabla_x F(t,u)$ on $[0,T]\times B_C$ and the boundedness of $\epsilon\|\dot{u}_\epsilon(t)\|$, the assertion follows.\\
(vii)-(viii) From the identity
\begin{align*}
\frac{1}{2\epsilon} \|\nabla_xF(r,u_\epsilon(r))+\epsilon^2A\ddot{u}_\epsilon(r)\|^2&= \frac{1}{2\epsilon}\|\nabla_xF(r,u_\epsilon(r))\|^2+\frac{\epsilon^3}{2}\|A\ddot{u}_\epsilon(r)\|^2\\
&+\epsilon\langle A\ddot{u}_\epsilon(r), \nabla_x F(r,u_\epsilon(r))\rangle,
\end{align*}
by integrating in time we get the estimate
\begin{align*}
& \displaystyle\frac{\epsilon^3}{2}\int_0^T\|A\ddot{u}_\epsilon(r)\|^2\,\mathrm{d}r+\frac{1}{2\epsilon}\int_0^T\|\nabla_xF(r,u_\epsilon(r))\|^2\,\mathrm{d}r\\
&\leq \frac{1}{2\epsilon} \int_0^T\|\nabla_xF(r,u_\epsilon(r))+\epsilon^2A\ddot{u}_\epsilon(r)\|^2\,\mathrm{d}r
+\epsilon\left|\int_0^T\langle A\ddot{u}_\epsilon(r), \nabla_x F(r,u_\epsilon(r))\rangle\,\mathrm{d}r\right|\,.
\end{align*}
From this, exploiting the coerciveness of $A$, both (vii) and (viii) immediately follow with (v) and (vi).

\endproof

\begin{cor}
Under assumptions of {\rm Proposition~\ref{apriori}}, it holds
\begin{equation*}
\epsilon^2\|\dot{u}_\epsilon(t)\|^2\to0,\, \text{as $\epsilon\to0$, }\,\text{ for almost every }\, t\in[0,T].
\end{equation*}
\label{a.e.}
\end{cor}

\proof
From Proposition~\ref{apriori}(iv) we obtain
\begin{equation*}
\displaystyle\epsilon^2\int_0^T\|\dot{u}_\epsilon(r)\|^2\,\mathrm{d}r\leq C\epsilon,
\end{equation*}
from which we deduce the convergence $\epsilon\dot{u}_\epsilon\to0$ in $L^2(0,T)$ and then the convergence a.e. in $[0,T]$.
\endproof

A useful tool in the proof of the compactness Theorem~\ref{compact} will be the following technical result dealing with the asymptotic behaviour of the energy dissipation integrals $\int_{s_k}^{t_k}\|\nabla_x F(r,v_{k}(r))\|\|\dot v_{k}(r)\|\mathrm{d}r$, where the curves $(v_{k})_k$ are defined on intervals $[s_k,t_k]$ shrinking to a point $\{t\}$ as $k\to+\infty$. More precisely, exploiting the assumption {\bf (F5)} on the isolatedness of the critical points of $\nabla_x F(t,\cdot)$, we show that if $v_{k}(s_k)\to u_1$, $v_{k}(t_k)\to u_2$, with $u_1\neq u_2$ , then the energy dissipation integrals are bounded away from zero. Notice that for the argument below one does not need to require that $u_1,u_2$ belong to $\mathcal{C}(t)$

\begin{prop}
Assume {\bf (F0)-(F3)} and {\bf(F5)}. Let $t\in[0,T]$, $u_1,u_2\in X$ and $s_k,t_k$ be sequences such that $0\leq s_k\leq t_k\leq T$ for every $k\in\N$ and $s_k\to t$, $t_k\to t$ as $k\to+\infty$. Let $(v_{k})_k$ be such that $v_{k}(t_k)\to u_1$ and $v_{k}(s_k)\to u_2$ as $k\to+\infty$, with $u_1\neq u_2$. Assume that there exists $M>0$ such that $\|v_{k}(r)\|\leq M$ for every $r\in[s_k,t_k]$ and for every $k\in\N$.
Then there exist $\delta=\delta(t,M, u_1,u_2)>0$ and $k_0\in\N$ such that
\begin{equation}
\int_{s_k}^{t_k}\|\nabla_xF(r,v_{k}(r))\|\|\dot v_{k}(r)\|\mathrm{d}r>\delta,\quad\text{for every }k\geq k_0.
\label{propdelta1}
\end{equation}
\label{propdelta}
\end{prop}

\proof
By assumptions {\bf (F1)} and {\bf (F5)}, the set $B_M\cap \mathcal{C}(t)$ is finite, thus there exists $\bar{\eta}=\bar{\eta}(t,M, u_1,u_2)$ such that, for every $0<\eta\leq \bar{\eta}$, it holds
\begin{equation}
B_{2\eta}(v)\cap B_{2\eta}(w)=\emptyset, \quad \text{ for every }v,w\in\mathcal{S},\, v\neq w,
\label{dupalle}
\end{equation}
where $\mathcal{S}=\mathcal{S}(t,u_1,u_2,M):=(B_M\cap \mathcal{C}(t))\cup\{u_1,u_2\}$. Now, if we introduce the compact set $K_\eta$ defined by $K_\eta:=B_M\backslash \bigcup_{v\in \mathcal{S}}B_\eta(v)$, we have that $\displaystyle\min_{u\in K_\eta}\|\nabla_x F(t,u)\|>0$ and, by the regularity assumption {\bf (F0)}, there exists $\gamma=\gamma(t,M, u_1,u_2)>0$ such that
\begin{equation}
m_\eta:=\min_{u\in K_\eta, r\in[t-\gamma,t+\gamma]} \|\nabla_x F(r,u)\|>0.
\label{mineta}
\end{equation}
Since $t_k\to t$ and $s_k\to t$, for every $k$ sufficiently large we have that $[s_k,t_k]\subset [t-\gamma,t+\gamma]$. Moreover, since $v_{k}(t_k)\to u_1$ and $v_{k}(s_k)\to u_2$, and from the definition of $K_\eta$ we also get that the set
\begin{equation*}
\mathcal{T}_k:=\{r\in[s_k,t_k]:\,v_{k}(r)\in K_\eta\}
\end{equation*}
is nonempty, for $k$ sufficiently large, and that there exist $r_1,r_2\in \mathcal{T}_k$, with $r_1\neq r_2$, such that $\|v_{k}(r_1)-u_1\|=\eta$ and $\|v_{k}(r_2)-u_2\|=\eta$. Thus, by \eqref{mineta} we get
\begin{align*}
\int_{s_k}^{t_k}\|\nabla_xF(r,v_{k}(r))\|\|\dot v_{k}(r)\|\mathrm{d}r&\geq \int_{\mathcal{T}_k}\|\nabla_xF(r,v_{k}(r))\|\|\dot v_{k}(r)\|\mathrm{d}r\\
& \geq m_\eta \int_{\mathcal{T}_k}\|\dot v_{k}(r)\|\mathrm{d}r\\
&\geq m_\eta \min_{v,w\in\mathcal{S}}(\|v-w\|-2\eta)=:\delta,
\end{align*}
where $\delta$ is strictly positive by \eqref{mineta} and \eqref{dupalle}.
\endproof

\section{Compactness}\label{compactness}

The main result of this section is the following compactness result. We will show that $u_\epsilon(t)$ converges, as $\epsilon\to0$ along a subsequence independent of $t$, to a regulated function $u(t)$ for all $t$. This limit function satisfies the stability condition
\begin{equation}
\nabla_x F(t,u(t))=0
\label{staz}
\end{equation}
for all $t \in [0,T]\setminus J$, where this latter is the (at most countable) jump set of $u$. Moreover, at each jump point $t\in J$, it holds
\begin{equation}
\nabla_x F(t,u_-(t))=\nabla_x F(t,u_+(t))=0.
\end{equation}

\begin{thm}[Compactness]\label{compact} Assume that {\bf (F0)-(F5)} hold and let ${u}_{\epsilon}: [0,T]\longrightarrow X$ be the solution of the Cauchy problem associated to \eqref{mainequation} with initial condition at $t=0$ and $u_\epsilon(0)$, $\epsilon \dot{u}_\epsilon(0)$ be uniformly bounded as $\epsilon\to0$. Then, up to a subsequence independent of $t$, $(u_\epsilon)_\epsilon$ converge pointwise, as $\epsilon\to0$, to a function $u:[0,T]\longrightarrow X$ satisfying the following properties:
\begin{enumerate}
\item[\rm(i)] $u$ is regulated;
\item[\rm(ii)] it holds
\begin{equation}
\nabla_x F(t,u_+(t))=\nabla_xF(t,u_-(t))=0 \quad \text{in $X$ for every $t\in (0,T]$}\,,
\label{(4.1)}
\end{equation}
where we understand $u_+(T):=u(T)$.
\end{enumerate}

\end{thm}

\proof
Throughout the proof, $M$ will denote an upper bound for $\|u_\epsilon(\cdot)\|$, whose existence is proved in Proposition \ref{apriori} (i).
We consider the family of positive measures $(\mu_\epsilon)_\epsilon$ defined as
\begin{equation*}
\mu_\epsilon:=\epsilon\|\dot{u}_\epsilon(\cdot)\|_B^2\mathcal{L}^1,
\end{equation*}
where $\mathcal{L}^1$ denotes the 1-dimensional Lebesgue measure. From Proposition~\ref{apriori}(iv), the family $(\mu_\epsilon)_\epsilon$ is equibounded in $L^1(0,T)$, therefore it converges weakly* (up to a subsequence) to a positive finite measure $\mu$ on $[0,T]$. Then, the set of atoms $J_\mu$ of $\mu$ is  at most countable. As a consequence of Proposition~\ref{apriori}(vii), we have also that
\begin{equation}
\nabla_x F(t, u_\epsilon(t))\to 0\quad\mbox{ for a.e. }t\in [0,T]\,.
\label{eq3}
\end{equation}

We may now fix a countable dense subset $I \subset [0,T]$ with the property that $I \supset J_\mu \cup\{0\}$, and define for all $t\in I$ the pointwise limit $u(t)$ of $u_\epsilon(t)$ (along a time independent subsequence) via a diagonal argument.
If $t \in [0,T]\setminus I$,  it holds in particular $t \notin J_\mu$. Let $(t_k)_k$ and $(s_k)_k$ be two distinct sequences of points in the set $I$, both converging to $t$, and $u_1$ and $u_2$ be the limits of $u(t_k)$ and $u(s_k)$, respectively. With a diagonal procedure we can  extract a subsequence $u_{\epsilon_k}$ such that
\[
u_{\epsilon_k}(t_k)\to u_1\quad\mbox{ and }\quad u_{\epsilon_k}(s_k)\to u_2\,.
\] 
Up to further extraction, it holds either $t_k\le s_k$ or $s_k \le t_k$ for all $k$. Assuming this last one is the case, we then have  by the convergence of $\mu_{\epsilon_k}$ to $\mu$, namely, by the upper semicontinuity, that
\begin{equation*}
\mu([t-\eta, t+\eta])\ge  \limsup_{k\to+\infty}\left(\epsilon_k \int_{s_k}^{t_k}\|\dot u_{\epsilon_k}(r)\|^2_B\mathrm{d}r\right),
\end{equation*}
for any $\eta>0$. Letting $\eta \to 0$, since $t \not\in J_\mu$ we deduce that
\begin{equation}
\lim_{k\to+\infty} \epsilon_k\int_{s_k}^{t_k}\|\dot u_{\epsilon_k}(r)\|^2\mathrm{d}r=0\,,
\label{(*)}
\end{equation}
where we additionally exploited that $\|\cdot\|_B$ is an equivalent norm.

Now, being $u_1\neq u_2$, by Proposition~\ref{propdelta} we may find $\delta=\delta(t,M, u_1,u_2)>0$ such that
\begin{equation*}
\int_{s_k}^{t_k}\|\nabla_xF(r,u_{\epsilon_k}(r))\|\|\dot u_{\epsilon_k}(r)\|\mathrm{d}r>\delta,
\end{equation*}
for $k$ large enough. Then, as a consequence of Young's inequality and of Proposition~\ref{apriori}(viii), we obtain
\begin{equation}
\begin{split}
\delta &< \int_{s_k}^{t_k}\|\nabla_xF(r,u_{\epsilon_k}(r))\|\|\dot u_{\epsilon_k}(r)\|\mathrm{d}r\\
&\leq \int_{s_k}^{t_k}\|\nabla_xF(r,u_{\epsilon_k}(r))+\epsilon_k^2A\ddot{u}_\epsilon(r)\|\|\dot u_{\epsilon_k}(r)\|\mathrm{d}r
+ \int_{s_k}^{t_k}\left(\epsilon_k^\frac{3}{2}\|A\ddot{u}_{\epsilon_k}(r)\|\right)\left(\epsilon_k^{\frac{1}{2}}\|\dot u_{\epsilon_k}(r)\|\right)\mathrm{d}r\\
&\stackrel{\eqref{mainequation}}{=}  \epsilon_k\int_{s_k}^{t_k}\|B\dot u_{\epsilon_k}(r)\|\|\dot u_{\epsilon_k}(r)\|\mathrm{d}r
+ \int_{s_k}^{t_k}\left(\epsilon_k^\frac{3}{2}\|A\ddot{u}_{\epsilon_k}(r)\|\right)\left(\epsilon_k^{\frac{1}{2}}\|\dot u_{\epsilon_k}(r)\|\right)\mathrm{d}r\\
&{\leq} \epsilon_k\|B\|\int_{s_k}^{t_k}\|\dot u_{\epsilon_k}(r)\|^2\mathrm{d}r + \frac{\delta\|A\|^2}{2C\|A\|^2}\int_{s_k}^{t_k}\epsilon_k^3\|\ddot{u}_{\epsilon_k}(r)\|^2\,\mathrm{d}r + \frac{C\|A\|^2}{2\delta}\int_{s_k}^{t_k}\epsilon_k\|\dot{u}_{\epsilon_k}(r)\|^2\,\mathrm{d}r\\
&\leq \left(\|B\|+\frac{C\|A\|^2}{2\delta}\right)\int_{s_k}^{t_k}\epsilon_k\|\dot{u}_{\epsilon_k}(r)\|^2\,\mathrm{d}r + \frac{\delta}{2},
\end{split}
\label{bigstima}
\end{equation}
whence, passing to the limit as $k\to+\infty$ and with \eqref{(*)} we deduce $\delta\leq\delta/2$, which is a contradiction. Therefore, it must be $u_1=u_2$.

Setting $u(t)=u_1$ we can then extend $u$ in a unique way to a function defined on all $[0,T]$. Furthermore, the same argument as above, with $u_1=u(t)$ and $(s_k)_k$ being the sequence constantly equal to $t$, together with the Urysohn's property, shows that $u_\epsilon(t)$ converge to $u(t)$ also for $t\in [0,T]\backslash I$. A further application of the same argument shows that $u$ is continuous at any $t\in [0,T]\backslash J_\mu$. Therefore, the jump set $J$ of $u$ is contained in $J_\mu$, and is at most countable. By pointwise convergence and \eqref{eq3}, we also have $\nabla_x F(t, u(t))= 0$ for almost every $t\in [0,T]$. By continuity, we deduce
\begin{equation*}
\nabla_x F(t, u(t))= 0
\end{equation*}
for every $t \in [0,T]\setminus J$.

To prove the existence of the left and right limits of $u$, we fix two sequences $(t_k)_k$ and $(s_k)_k$ with $t_k$, $s_k\searrow t$. It is not restrictive to assume that $s_k\leq t_k$, for all $k$. In order to prove the existence of $u_+(t)$, it will suffice to show that
\begin{equation*}
\displaystyle\lim_{k\to+\infty} \|u(s_k)-u(t_k)\|=0.
\end{equation*}
For this, we argue by contradiction and assume that
\begin{equation*}
\displaystyle\mathop{\lim\inf}_{k\to+\infty} \|u(s_k)-u(t_k)\|>0.
\end{equation*}
Up to extracting a subsequence (not relabeled), we may assume also that
\begin{equation*}
u(t_k)\to u_1\,\mbox{ and }\,u(s_k)\to u_2
\end{equation*}
with $u_1\neq u_2$. Now, setting
\begin{equation*}
g_\epsilon(t):= F(t,u_\epsilon(t))+\frac{\epsilon^2}{2}\|\dot{u}_\epsilon(t)\|^2_A-\int_0^t \partial_rF(r,u_\epsilon(r))\,\mathrm{d}r,
\end{equation*}
by \eqref{unostima} it follows that the functionals $t\longmapsto g_\epsilon(t)$ are non-increasing and bounded on $[0,T]$. Therefore, as a consequence of Helly's Theorem, there exists $g\in {\rm BV}([0,T])$ such that, up to a subsequence (not relabeled), $g_\epsilon(t)\to g(t)$ for every $t\in[0,T]$, where $g$ is non-increasing.  

We then have
\begin{equation*}
\begin{split}
&\displaystyle\lim_{k\to+\infty}\left(\lim_{\epsilon\to0}g_\epsilon(t_k)\right)=\lim_{k\to+\infty}\left(\lim_{\epsilon\to0}g_\epsilon(s_k)\right)=g_+(t)\,,\\
&\displaystyle\lim_{k\to+\infty}\left(\lim_{\epsilon\to0}u_\epsilon(t_k)\right)=u_1\,,\\
&\displaystyle\lim_{k\to+\infty}\left(\lim_{\epsilon\to0}u_\epsilon(s_k)\right)=u_2\,.
\end{split}
\end{equation*}
With a diagonal procedure we can  extract a subsequence $\epsilon_k\to0$ such that
\begin{align}
&\displaystyle\lim_{k\to+\infty}g_{\epsilon_k}(t_k)=\lim_{k\to+\infty}g_{\epsilon_k}(s_k)=g_+(t)\,,\label{(1)}\\
&\displaystyle\lim_{k\to+\infty}u_{\epsilon_k}(t_k)=u_1\neq u_2=\displaystyle\lim_{k\to+\infty}u_{\epsilon_k}(s_k)\,.\label{(2)}
\end{align}

Now, by using the energy identity \eqref{zerostima} and \eqref{(1)} we get
\begin{equation*}
\displaystyle\mathop{\lim\sup}_{k\to+\infty}\int_{s_k}^{t_k}\epsilon_k\|\dot{u}_{\epsilon_k}(r)\|_B^2\,\mathrm{d}r \leq \displaystyle\mathop{\lim\sup}_{k\to+\infty} \left(g_{\epsilon_k}(s_k)-g_{\epsilon_k}(t_k)\right)=g_+(t)-g_+(t)=0.
\end{equation*}
Again,  since $\|\cdot\|_B$ is an equivalent norm, we deduce
\begin{equation}
\displaystyle\lim_{k\to+\infty}\int_{s_k}^{t_k}\epsilon_k\|\dot{u}_{\epsilon_k}(r)\|^2\,\mathrm{d}r =0.
\label{(4)}
\end{equation}

Assuming \eqref{(2)} and with \eqref{(4)}, we can perform an analogous argument as in the proof of the continuity of $u$ on $[0,T]\backslash J$. Namely, since $u_1\neq u_2$, we may find $\delta=\delta(t,M, u_1,u_2)>0$ such that
\begin{equation*}
\int_{s_k}^{t_k}\|\nabla_xF(r,u_{\epsilon_k}(r))\|\|\dot u_{\epsilon_k}(r)\|\mathrm{d}r>\delta,
\end{equation*} 
for $k$ large enough. Then, by arguing as for \eqref{bigstima}, we finally get
\begin{equation*}
\delta < \int_{s_k}^{t_k}\|\nabla_xF(r,u_{\epsilon_k}(r))\|\|\dot u_{\epsilon_k}(r)\|\mathrm{d}r\leq\left(\|B\|+\frac{C\|A\|^2}{2\delta}\right)\int_{s_k}^{t_k}\epsilon_k\|\dot{u}_{\epsilon_k}(r)\|^2\,\mathrm{d}r + \frac{\delta}{2},
\end{equation*}
from which, passing to the limit as $k\to+\infty$, we get a contradiction. Thus, $u_1=u_2$ and this proves the existence of $u_+(t)$ at every $t$. The existence of $u_-(t)$ can be proved along the same lines. The proof of (i) is then concluded. Now, (ii) immediately follows by \eqref{staz} and {\bf (F0)}.

\endproof

\begin{oss}
The point $t=0$ has in general to be excluded from \eqref{(4.1)} since, if we convene to denote $u_-(0):=u(0)$, the initial data $u_\epsilon(0)$ might not converge to  a critical point of $F(0,\cdot)$. Notice, instead, that by \eqref{(4.1)} the right-limit $u_+(0)$ must be a stationary point. Thus, if $u_-(0)\notin \mathcal{C}(0)$, a jump occurs already at the initial time. Also for such a jump we will provide a characterization in terms of a dissipation cost in the next Section.
\end{oss}

\section{Energy balance}\label{energybalance}

Our main aim is to show that the limit evolution $u(t)$ provided by Theorem~\ref{compact} satisfies a balance between the stored energy and  the power spent  along the evolution in an interval of time $[s, t]\subset [0, T]$, up to a (positive) dissipation cost which is concentrated on the jump set of $u$.

Define the function $f: [0,T] \longrightarrow \R$ as
\begin{equation}
f(t):= F(t,u(t))-\int_0^t \partial_rF(r,u(r))\,\mathrm{d}r,
\label{decr}
\end{equation}
and, correspondingly, consider its right-continuous and left-continuous representatives, namely
\begin{equation*}
f_+(t):= F(t,u_+(t))-\int_0^t \partial_rF(r,u(r))\,\mathrm{d}r,
\end{equation*}
and 
\begin{equation*}
f_-(t):= F(t,u_-(t))-\int_0^t \partial_rF(r,u(r))\,\mathrm{d}r,
\end{equation*}
respectively.

We note at first that, under our assumptions, and in particular by assuming {\bf (F6)} and {\bf (F3')}, the right continuous representative of function $f$ defined in \eqref{decr} has a positive right derivative at every point.

\begin{prop}
Assume {\bf (F0)-(F6)}, with  {\bf (F3')} in place of {\bf (F3)}. Let $f:[0,T]\longrightarrow\R$ be defined as in \eqref{decr}.
Then, for any $t\in[0,T]$, the Dini lower right derivative of the right-continuous representative $f_+$ at $t$ is non-negative; i.e., it holds
\begin{equation}
D_+{f_+}(t):=\mathop{\lim\inf}_{h\searrow0}\frac{f_+(t+h)-f_+(t)}{h}\geq0.
\label{liminfpos}
\end{equation}
\label{proposition}
\end{prop}

\begin{proof}
The proof is {\it verbatim} the one given in \cite[Proposition~5.1]{Sci-Sol}, since it only makes use of the stability condition \eqref{(4.1)} together with the assumptions on the potential $F$.
\end{proof} 

In order to prove our main result, we need the following elementary lemma, whose proof can be found in \cite[Lemma~5.2]{Sci-Sol}.

\begin{lem}
Let $g:[a,b]\longrightarrow\R$ be continuous and such that the \emph{Dini upper right derivative of $g$ at $t$},
\begin{equation*}
D^+g(t):=\mathop{\lim\sup}_{h\searrow0}\frac{g(t+h)-g(t)}{h}\geq0,\quad \forall t\in(a,b). 
\end{equation*}
Then $g$ is non-decreasing on $(a,b)$.
\label{lem}
\end{lem}

In order to pass to the limit in some energy inequalities, we will need the following Lemma, whose elementary proof is omitted.

\begin{lem}\label{trivial}
Assume {\bf(F0)}-{\bf(F5)}. Let $u_\epsilon$ and $u$ be defined as in \emph{Theorem~\ref{compact}}. Then the following results hold:
\begin{enumerate}
\item[\rm(i)]$\displaystyle\lim_{\epsilon\to0}F(t,u_\epsilon(t))=F(t,u(t))$, for every $t\in[0,T]$;
\item[\rm(ii)]$\displaystyle\lim_{\epsilon\to0}\int_0^t \partial_rF(r,u_\epsilon(r))\,\mathrm{d}r=\int_0^t \partial_rF(r,u(r))\,\mathrm{d}r$, for every $t\in[0,T]$.
\end{enumerate}
\end{lem}


The following theoretical result provides the energy balance equality \eqref{measure} for our energies $F(t,u)$.

\begin{thm}\label{genbalance}
Assume {\bf (F0)-(F6)}, with {\bf (F3')} in place of {\bf (F3)}. Let $u_\epsilon$ and $u$ be defined as in \emph{Theorem \ref{compact}}, with $\epsilon\dot{u}_\epsilon(0)\to0$, and let $J$ be the jump set of $u$. There exists a positive atomic measure $\mu$, with ${\rm supp}(\mu)=J$, such that
\begin{equation}
F(t,u_+(t))+\mu([s,t])=F(s,u_-(s))+\int_s^t \partial_r F(r,u(r))\,dr,
\label{measure}
\end{equation} 
for all $0\leq s\leq t\leq T$.
\end{thm}

\proof

Let $f: [0,T]\longrightarrow\mathbb{R}$ be defined as in \eqref{decr}. With fixed $\delta>0$, we extend the function $f$ to the open interval $(-\delta, T+\delta)$ as
\begin{equation*}
\tilde{f}(s)=
\begin{cases}
f(0) & \text{if }s\in(-\delta,0)\\
f(s) & \text{if }s\in[0,T]\\
f(T) & \text{if }s\in(T, T+\delta).
\end{cases}
\end{equation*}
With a slight abuse of notation, we will denote still by $f$ such an extension. Now, if we define $\mu:=-Df$, we have ${\rm supp}(\mu)\subseteq[0,T]$. We aim to show that $\mu$ is a positive measure, and for this it will suffice to prove that $f$ belongs to the Lebesgue class of a non-increasing function. First, we prove the following Claim.

\emph{Claim: $f_+$ is non-increasing.}\\
\emph{Proof of Claim.} Let $s,t\in(-\delta, T+\delta)$ with $s>t$, we then have three cases:
\begin{enumerate}
\item[(a)] $0<t<T$. We can fix two sequences $s_k\searrow s$, $t_k\searrow t$ with $s_k>t_k$ for every $k\in\N$ and such that $\epsilon^2\|\dot{u}_\epsilon(t_k)\|^2\to0$ (as a consequence of Corollary~\ref{a.e.}). From the energy inequality \eqref{unostima} we get
\begin{equation*}
F(s_k, u_\epsilon(s_k)) \leq F(t_k, u_\epsilon(t_k))+\frac{\epsilon^2}{2}\|\dot{u}_\epsilon(t_k)\|^2_A+ \int_{t_k}^{s_k}\partial_rF(r,u_\epsilon(r))\,\mathrm{d}r
\end{equation*}
and then, passing to the limit as $\epsilon\to0$, by Lemma~\ref{trivial} we obtain
\begin{equation*}
F(s_k, u(s_k)) \leq F(t_k, u(t_k))+ \int_{t_k}^{s_k}\partial_rF(r,u(r))\,\mathrm{d}r.
\end{equation*}
Finally, passing to the limit as $k\to+\infty$, we get
\begin{equation*}
F(s, u_+(s)) \leq F(t, u_+(t))+ \int_{t}^{s}\partial_rF(r,u(r))\,\mathrm{d}r,
\end{equation*}
which corresponds to $f_+(s)\leq f_+(t)$.
\item[(b)] $t<0$. If also $s<0$, the assertion is trivial. If $s\geq0$, we have to show that $f_+(s)\geq f_+(t)=0$. Since $\epsilon\dot{u}_\epsilon(0)\to0$, it will suffice to consider $s_k\searrow s$ and use the inequality
\begin{equation*}
F(s_k, u_\epsilon(s_k)) \leq F(0, u_\epsilon(0))+ \frac{\epsilon^2}{2}\|\dot{u}_\epsilon(0)\|^2_A + \int_{0}^{s_k}\partial_rF(r,u_\epsilon(r))\,\mathrm{d}r,
\end{equation*}
where we pass to the limit as $\epsilon\to0$ first and then as $k\to+\infty$.
\item[(c)] If $t\geq T$, since by convention $u_+(T)=u(T)$ and then $f_+(T)=f(T)$, the assertion follows immediately.
\end{enumerate}
This concludes the proof of the Claim.

The Claim implies now that $\mu=-Df$ is a positive measure and, in particular, $f\in BV(-\delta,T+\delta)$. Moreover, by \cite[Theorem~3.28]{AFP} we get
\begin{equation}
F(t,u_+(t))+\mu([s,t])=F(s,u_-(s))+\int_s^t \partial_r F(r,u(r))\,\mathrm{d}r,
\label{measurebis}
\end{equation}
for all $0\leq s\leq t$, with the usual convention that $u_+(T)=u(T)$ and $u_-(0)=u(0)$.

We are left to show that ${\rm supp}(\mu)=J$. In order to do that, we define
\begin{equation*}
f^J(t):=\sum_{s\in[0,t]} (f_+(s)-f_-(s)),
\end{equation*}
which is the right-continuous jump function of $f$. We note that the set of discontinuities of $f^J$ coincides with $J$ and $Df^J=(Df)^J$, the latter being the jump part of measure $Df$. Moreover, $f^J$ is nonincreasing, so that 
\begin{align}\label{fatto}
\mathop{\lim\inf}_{h\searrow0}\frac{f^J(t)-f^J(t+h)}{h}\geq0
\end{align}
and $\mu^J=-Df^J$ is positive. It holds also $\mu\geq \mu^J$, since $\mu$ is positive.
Summing up \eqref{fatto} and \eqref{liminfpos} we get
\begin{equation*}
\mathop{\lim\inf}_{h\searrow0}\frac{(f_+-f^J)(t+h)-(f_+-f^J)(t)}{h}\geq0.
\end{equation*}
Since, by construction, $f_+-f^J$ is a continuous function,  by Lemma~\ref{lem} $f_+-f^J$ is nondecreasing. Therefore $f_+(t)-f^J(t)\geq f_+(0)-f^J(0)$ for all $t\in[0,T]$, or, equivalently, 
\begin{equation*}
F(t,u_+(t))+\mu^J([0,t])\geq F(0,u(0))+\int_0^t \partial_r F(r,u(r))\,dr,
\end{equation*}
where, by the usual convention, $u(0)=u_-(0)$. Comparing the latter estimate with \eqref{measurebis} we finally get
\begin{equation*}
\mu^J([0,t])\geq\mu([0,t])\geq\mu^J([0,t]),\quad \forall t,
\end{equation*}
which gives $\mu^J=\mu$, thus concluding the proof.

\endproof

\begin{oss}
We note that, by construction, it holds
\begin{equation}\label{mu}
F(t,u_-(t))-F(t,u_+(t))=\mu(\{t\})>0,\quad \forall t\in J.
\end{equation}
\end{oss}

\subsection{The energy-dissipation cost}

In this section, we prove that the gap of the potential $F(t,u(t))$ at a jump point $t\in J$ can be measured by means of a (positive and symmetric) cost function, solution to an optimization problem with boundary conditions at infinity $v(-\infty)=u_-(t)$, $v(+\infty)=u_+(t)$; namely,
\begin{equation*}
F(t,u_-(t))-F(t,u_+(t))=c_t(u_-(t),u_+(t)),\quad\text{ for every }t\in J.
\end{equation*}
In order to lighten the notation, from now on we set $F_t(u):=F(t,u)$. 

\begin{defn}
For every $t\in[0,T]$ and $u_1,u_2\in X$, we define the \emph{energy-dissipation cost} as
\begin{equation}
c_t(u_1,u_2):=
\inf\left\{\frac{1}{2}\int_{-N}^{N}\|\nabla_xF_t(v(s))+A\ddot{v}(s)\|_{B^{-1}}^2+\|\dot{v}(s)\|_B^2\,\mathrm{d}s:\,v\in V^{t, N}_{u_1,u_2},\,N\in \mathbb{N}\right\}
\label{costfun}
\end{equation}
where 
\begin{equation*}
V^{t,N}_{u_1,u_2}:=\Bigl\{v\in W^{2,2}([-N, N],X):\, v(-N)=u_1, v(N)=u_2, \dot v(-N)=\dot v(N)=0\Bigr\}
\end{equation*}
denotes the class of the \emph{admissible curves}.
\end{defn}
The following theorem collects the main properties of the cost function.

\begin{thm}\label{cost}
Under assumptions {\bf (F0)-(F5)}, for every $t\in[0,T]$ and $u_1,u_2\in X$ we have:
\begin{enumerate}
\item[\rm (1)] the cost is symmetric; i.e., $c_t(u_1,u_2)=c_t(u_2,u_1)$;
\item[\rm (2)] $c_t(u_1,u_2)=0$ if and only if $u_1=u_2$;
\item[\rm (3)] $c_t(u_1,u_2)\geq |F_t(u_1)-F_t(u_2)|$;
\item[\rm (4)] for every $u_3\in X$, the triangle inequality
\begin{equation}
c_t(u_1,u_2)\leq c_t(u_1,u_3)+c_t(u_3,u_2)
\end{equation}
holds.
\item[\rm (5)] if $u_1 \neq u_2$ and $u_1, u_2 \in \mathcal C(t)$, then
\begin{equation}
c_t(u_1, u_2)=\inf\left\{\frac{1}{2}\int_{-\infty}^{+\infty}\|\nabla_xF_t(v(s))+A\ddot{v}(s)\|_{B^{-1}}^2+\|\dot{v}(s)\|_B^2\,\mathrm{d}s:\,v\in V^t_{u_1,u_2}\right\}
\label{costfun+}
\end{equation}
where 
\begin{equation*}
V^{t}_{u_1,u_2}:=\Bigl\{v\in W^{2,2}(\R,X):\, v(-\infty)=u_1, v(+\infty)=u_2\Bigr\}\,.
\end{equation*}
\end{enumerate}
\end{thm}

\proof
(1) If $u_1=u_2$, then the assertion is trivial by the definition of $c_t(u_1,u_2)$. Thus, let $u_1\neq u_2$, fix $N \in \mathbb{N}$, consider $v\in V^{t,N}_{u_1,u_2}$ and define $\tilde{v}(s):=v(-s)$. We then have $\tilde{v}\in V^{t,N}_{u_2,u_1}$ and    
\begin{equation}
\begin{split}
c_t(u_2,u_1)&\leq
\frac{1}{2}\int_{-N}^{N}\|\nabla_xF_t(\tilde{v}(s))+A\ddot{\tilde{v}}(s)\|_{B^{-1}}^2+\|\dot{\tilde{v}}(s)\|_B^2\,\mathrm{d}s\\
&= \frac{1}{2}\int_{-N}^{N}\|\nabla_xF_t(v(s))+A\ddot{v}(s)\|_{B^{-1}}^2+\|\dot{v}(s)\|_B^2\,\mathrm{d}s \\.
\end{split}
\label{relazione1}
\end{equation}
Now, taking the infimum in the right hand side of \eqref{relazione1} on the set $V^{t,N}_{u_2,u_1}$ and $N \in \mathbb{N}$, we get the inequality
\begin{equation*}
c_t(u_2,u_1)\leq c_t(u_1,u_2).
\end{equation*} 
The assertion then follows by interchanging the role of $u_1$ and $u_2$.
\\
(2) Let $u_1\neq u_2$. It clearly suffices to provide a lower bound, independent of $N$, on the energy of competitors $v \in V^{t,N}_{u_1,u_2}$ satisfying
\begin{equation}
\frac12\int_{-N}^{N}\|\nabla_xF_t(v(s))+A\ddot{v}(s)\|_{B^{-1}}^2+\|\dot{v}(s)\|_B^2\,\mathrm{d}s\leq 1\,.
\label{bounduni}
\end{equation}
We begin by proving the following\\
\emph{Claim:} there exists  $M>0$, not depending on $N$, such that
\begin{align}
&\sup_{s\in[-N, N]}\|v(s)\|+\|\dot{v}(s)\|_A^2\le M \label{equa1n}\\
&\int_{-N}^{N}\|A\ddot{v}(s)\|^2\,\mathrm{d}s\leq M,\label{equa2}\\
&\int_{-N}^{N}\|\nabla_xF_t(v(s))\|^2\,\mathrm{d}s\leq M,\label{equa3}
\end{align}
for every $v\in V^{t,N}_{u_1,u_2}$ which satisfies \eqref{bounduni}.

In order to prove the Claim, we note that, by the fundamental theorem of calculus, with the Cauchy-Schwarz inequality \eqref{CSequi} and \eqref{bounduni}, for every $s\in[-N, N]$ we have
\begin{equation}
\begin{split}
&F_t(v(s))+\frac{1}{2}\|\dot{v}(s)\|_A^2=F_t(u_1)+\int_{-N}^s \langle \nabla_x F_t(v(r))+A\ddot{v}(r), \dot{v}(r)\rangle\,\mathrm{d}r\\
&\leq F_t(u_1)+\frac{1}{2}\int_{-N}^s \|\nabla_x F_t(v(r))+A\ddot{v}(r)\|_{B^{-1}}^2+ \|\dot{v}(r)\|_B^2\,\mathrm{d}r\\
&\leq  F_t(u_1)+1\,.
\label{(5.22)}
\end{split}
\end{equation}
Now, from \eqref{(5.22)} we get the equi-boundedness of $F_t(v(s))$ and $\|\dot{v}(s)\|_A^2$, and, in view of {\bf (F1)}, also \eqref{equa1n}. Furthermore, denoting with $\alpha$ and $\beta$ the coercivity constants of $A$ and $B$, respectively, with \eqref{equa1n} and an integration by parts  we obtain
\begin{equation*}
\begin{split}
&\int_{-N}^{N} \|A\ddot{v}(s)\|^2\,\mathrm{d}s + \int_{-N}^{N} \|\nabla_x F_t(v(s))\|^2\,\mathrm{d}s\\
&=\int_{-N}^{N} \|\nabla_xF_t(v(s))+A\ddot{v}(s)\|^2\,\mathrm{d}s-2\int_{-N}^N \langle \nabla_xF_t(v(s)), A\ddot{v}(s)\rangle\,\mathrm{d}s\\
&\leq \frac2{\beta^2}+2\left|\int_{-N}^N\langle \nabla_x^2F_t(v(s))A\dot{v}(s),\dot{v}(s)\rangle\,\mathrm{d}s\right|
\\&\leq\frac2{\beta^2}+\frac {M\alpha}{\beta^2}\int_{-N}^{N}\|\dot{v}(s)\|_B^2\,\mathrm{d}s\leq \frac{2+M\alpha}{\beta^2}
\end{split}
\end{equation*} 
where we have also exploited the regularity assumptions {\bf(F0)} and {\bf(F4)}, and \eqref{bounduni}. Since the constant on the right-hand side does not depend on $N$, this gives \eqref{equa2}-\eqref{equa3} and concludes the proof of Claim.

We now observe that, since $u_1 \neq u_2$, we may find a constant $\delta=\delta(t,u_1,u_2)>0$, not depending on $N$, such that
\begin{equation}
\delta<\int_{-N}^{N}\|\nabla_xF_t(v(s))\|\|\dot{v}(s)\|\,\mathrm{d}s,\quad \text{ for every }v\in V^{t,N}_{u_1,u_2}. 
\label{equa1}
\end{equation}
This can be proved along similar lines than in the proof of \eqref{propdelta1}, using \eqref{equa1n} and {\bf (F5)}.
Taking into account \eqref{equa1}, \eqref{equa3} and applying the Cauchy inequality we further get
\begin{equation*}
\begin{split}
\delta&<\int_{-N}^{N}\|\nabla_xF_t(v(s))\|\|\dot{v}(s)\|\,\mathrm{d}s\leq\int_{-N}^{N}\sqrt{\frac{\delta}{M}}\|\nabla_xF_t(v(s))\|\sqrt{\frac{M}{\delta}}\frac{1}{\beta}\|\dot{v}(s)\|_B\,\mathrm{d}s\\
&\leq \frac{\delta}{2M}\int_{-N}^{N}\|\nabla_xF_t(v(s))\|^2\,\mathrm{d}s + \frac{M}{2\delta\beta^2}\int_{-N}^{N}\|\dot{v}(s)\|_B^2\,\mathrm{d}s\\
&\leq \frac{\delta}{2}+\frac{M}{2\delta\beta^2}\int_{-N}^{N}\|\dot{v}(s)\|_B^2\,\mathrm{d}s\\
&\leq\frac{\delta}{2}+\frac{M}{\delta\beta^2}\left(\frac{1}{2}\int_{-N}^{N}\|\nabla_xF_t(v(s))+A\ddot{v}(s)\|_{B^{-1}}^2+\|\dot{v}(s)\|_B^2\,\mathrm{d}s\right),
\end{split}
\end{equation*}
which gives
\begin{equation}
\frac{1}{2}\int_{-N}^{N}\|\nabla_xF_t(v(s))+A\ddot{v}(s)\|_{B^{-1}}^2+\|\dot{v}(s)\|_B^2\,\mathrm{d}s \geq\frac{\delta^2\beta^2}{2M}>0.
\label{newdelta}
\end{equation}
Since this lower bound is independent of the competitor $v$ and of $N$, we get $c_t(u_1,u_2)>0$ whenever $u_1 \neq u_2$. The other implication is obvious.
\\
(3) Since the cost $c_t(u_1,u_2)$ is symmetric, it will suffice to show that
\begin{equation*}
c_t(u_1,u_2)\geq F_t(u_2)-F_t(u_1).
\end{equation*} 
From the definition of $c_t(u_1,u_2)$, for every fixed $\eta>0$ there exist $N \in \mathbb{N}$ and $v\in V^{t,N}_{u_1,u_2}$ such that
\begin{equation*}
\frac{1}{2}\int_{-N}^{N}\|\nabla_xF_t(v(s))+A\ddot{v}(s)\|_{B^{-1}}^2+\|\dot{v}(s)\|_B^2\,\mathrm{d}s\leq c_t(u_1,u_2)+\eta.
\end{equation*}
With the fundamental theorem of calculus and \eqref{CSequi}, we then have
\begin{equation*}
\begin{split}
F_t(u_2)&-F_t(u_1)=F_t(v(N))-F_t(v(-N))+\frac12\left(\|\dot v(N)\|_A^2-\|\dot v(-N)\|_A^2\right)\\
&=\int_{-N}^{N}\langle\nabla_xF_t(v(s))+A\ddot{v}(s),\dot{v}(s)\rangle\,\mathrm{d}s\\
&\leq \frac{1}{2}\int_{-N}^{N}\|\nabla_xF_t(v(s))+A\ddot{v}(s)\|_{B^{-1}}^2+\|\dot{v}(s)\|_B^2\,\mathrm{d}s\leq c_t(u_1,u_2)+\eta,
\end{split}
\end{equation*}
whence we get the thesis by the arbitrariness of $\eta$.\\
(4) With $\eta>0$ fixed, we may find $N_1, N_2 \in \mathbb{N}$, $v_1\in V^{t, N_1}_{u_1,u_3}$ and $v_2\in V^{t, N_2}_{u_3,u_2}$ such that
\begin{align}
c_t(u_1,u_3)&\geq\frac{1}{2}\int_{-N_1}^{N_1}\|\nabla_xF_t(v_1(s))+A\ddot{v}_1(s)\|_{B^{-1}}^2+\|\dot{v}_1(s)\|_B^2\,\mathrm{d}s - \eta,\label{conda}\\
c_t(u_3,u_2)&\geq\frac{1}{2}\int_{-N_2}^{N_2}\|\nabla_xF_t(v_2(s))+A\ddot{v}_2(s)\|_{B^{-1}}^2+\|\dot{v}_2(s)\|_B^2\,\mathrm{d}s - \eta.\label{condb}
\end{align}
It suffices to set 
\[
v_3(s)=\left\lbrace
  \begin{array}{l}
    v_1(s+N_2)\quad\mbox{if }s \in [-(N_1+N_2), N_1-N_2] \\[5pt]
    v_2(s-N_1)\quad\mbox{if }s \in [N_1-N_2, N_1+N_2] \,,
  \end{array}
  \right.
\]
to have an admissible competitor in $V^{t, N_1+N_2}_{u_1, u_2}$. With \eqref{conda} and \eqref{condb} we get $c_t(u_1, u_2)\le c_t(u_1, u_3)+ c_t(u_3, u_2) +2\eta$, whence the conclusion follows by arbitrariness of $\eta$.
\\
(5) The ''$\ge$'' inequality is obvious, since each function  $v\in V^{t, N}_{u_1, u_2}$ can be extended to a function in $V^t_{u_1, u_2}$ by simply setting $v(s)=u_1$ if $s\le -N$ and $v(s)=u_2$ if $s \ge N$ without spending additional energy whenever $u_1, u_2 \in \mathcal{C}(t)$.
To prove the converse, we set
\[
\tilde{c}_t(u_1, u_2)=\inf\left\{\frac{1}{2}\int_{-\infty}^{+\infty}\|\nabla_xF_t(v(s))+A\ddot{v}(s)\|_{B^{-1}}^2+\|\dot{v}(s)\|_B^2\,\mathrm{d}s:\,v\in V^t_{u_1,u_2}\right\}\,.
\] 
We fix $\eta>0$ and $v \in V^t_{u_1, u_2}$ with
\[
\frac{1}{2}\int_{-\infty}^{+\infty}\|\nabla_xF_t(v(s))+A\ddot{v}(s)\|_{B^{-1}}^2+\|\dot{v}(s)\|_B^2\,\mathrm{d}s \le \tilde{c}_t(u_1, u_2)+\eta
\]
By definition of $V^t_{u_1, u_2}$, we may fix $a<0<b\in\mathbb{R}$ with the property
\begin{align}
\|\dot{v}(a)\|+\|v(a)-u_1\|+\|\dot{v}(b)\|+\|v(b)-u_2\| \leq\eta\label{cond4bis},
\end{align}
Next, we define the function $z$ as
\begin{equation*}
z(s)=
\begin{cases}
u_1+g(s+1-a)(v(a)-u_1)+h(s+1-a)\dot{v}(a), & \text{if }s\in(a-1,a],\\
v(s), & \text{if }s\in(a,b),\\
v(b)+g(s-b)(u_2-v(b))+\ell(s-b)\dot{v}(b), & \text{if }s\in[b, b+1),
\end{cases}
\end{equation*}
where 
\[
g(p)=3p^2-2p^3,\quad h(p)=p^2(p-1),\quad \ell(p)= p^3-2p^2+p,\quad  p\in[0,1].
\]
Note that, by construction, the right and left limits of both $z$ and $\dot z$ for $s \to a$ and $s \to b$ coincide, so that $z \in W^{2,2}([a-1, b+1],X)$. It also holds $\dot z(a-1)=\dot z(b+1)=0$ so that we simply can extend $z$ with the constant values $z(s)=u_1$ for $s\ge a-1$ and $z(s)=u_2$ for $s \ge b+1$ to have $z\in V^{t,N}_{u_1,u_2}$ for a suitable $N \in \mathbb{N}$. 

Taking into account  \eqref{cond4bis} we infer
\begin{equation*}
\frac12\int_{\R\setminus(a, b)}\|\nabla_xF_t(z(s))+A\ddot{z}(s)\|_{B^{-1}}^2+\|\dot{z}(s)\|_B^2\,\mathrm{d}s \leq C(\|A\|,\|B\|)\eta^2+\omega^2(\eta)\
\end{equation*}
where $\omega(\eta)$ is a modulus of continuity for $\nabla_xF_t(\cdot)$ on the union of balls $B_1(u_2)\cup B_1(u_1)$. With this,
we finally get
\[
\begin{split}
c_t(u_1, u_2) & \leq \frac12 \int_{-N}^N \|\nabla_xF_t(z(s))+A\ddot{z}(s)\|_{B^{-1}}^2+\|\dot{z}(s)\|_B^2\,\mathrm{d}s \\
&\leq C(\|A\|,\|B\|)\eta^2+\omega^2(\eta)+ \frac12 \int_a^b \|\nabla_xF_t(z(s))+A\ddot{z}(s)\|_{B^{-1}}^2+\|\dot{z}(s)\|_B^2\,\mathrm{d}s\\
&= C(\|A\|,\|B\|)\eta^2+\omega^2(\eta)+ \frac12 \int_a^b \|\nabla_xF_t(v(s))+A\ddot{v}(s)\|_{B^{-1}}^2+\|\dot{v}(s)\|_B^2\,\mathrm{d}s\\
& \leq C(\|A\|,\|B\|)\eta^2+\omega^2(\eta)+ \tilde{c}_t(u_1, u_2)+\eta\,,
\end{split}
\]
which implies the conclusion by arbitrariness of $\eta$.

\endproof

We can now show that $c_t(u_+(t),u_-(t))$ is a lower bound for the dissipation $\mu(\{t\})$ at a jump point $t$.
\begin{prop}
Assume {\bf (F0)-(F5)}. Let $c_t$ be the cost function defined in \eqref{costfun}, $u_-(t)$ and $u_+(t)$ be the left and right limits, respectively, of the function $u$ of {\rm Theorem~\ref{compact}} at each point $t$. Then it holds
\begin{equation}
F_t(u_-(t))-F_t(u_+(t))\geq c_t(u_+(t),u_-(t)),\quad \forall t\in (0,T].
\label{lowbounddiss}
\end{equation}
\label{proposizione1}
If additionally $\epsilon \dot u_\epsilon(0) \to 0$, the above inequality also holds for $t=0$.
\end{prop}

\proof 
Let $u_\epsilon, u$ be as in the statement of Theorem~\ref{compact}. We restrict to the case $t\in J$, since for any $t\in [0,T]\backslash J$ \eqref{lowbounddiss} holds as an equality in view of Theorem~\ref{cost}(2). If $t=0$ we  convene that the function $u_\epsilon$ is extended to a left neighborhood of $0$ with an affine function of constant slope $\dot u_\epsilon(0)$. First, we note that we can find sequences $s_k\nearrow t$ and $t_k\searrow t$ and a subsequence $\epsilon_k\to0$ such that
\begin{align}
&u_{\epsilon_k}(s_k)\to u_-(t),\, u_{\epsilon_k}(t_k)\to u_+(t),\label{hyp1}\\
&\epsilon_k\dot{u}_{\epsilon_k}(s_k)\to 0,\, \epsilon_k\dot{u}_{\epsilon_k}(t_k)\to 0,\label{hyp2}
\end{align}
as $k\to+\infty$. For this, we preliminary remark that, by virtue of Corollary~\ref{a.e.} (and of the assumption $\epsilon \dot u_\epsilon(0) \to 0$ in the case $t=0$), we can fix two sequences $s_k\nearrow t$ and $t_k\searrow t$ such that
\begin{equation}
\displaystyle\lim_{\epsilon\to0}\epsilon\dot{u}_\epsilon(s_k)=\lim_{\epsilon\to0}\epsilon\dot{u}_\epsilon(t_k)=0.
\label{ipo1}
\end{equation}
Moreover, since $u$ is regulated by Theorem~\ref{compact}(i), there also hold
\begin{equation}
u(s_k)\to u_-(t), \quad u(t_k)\to u_+(t).
\label{ipo2}
\end{equation}
Now, we define by induction on $k\geq1$ the sequence
\begin{equation*}
\epsilon_k:=\max\left\{\epsilon\leq\frac{\epsilon_{k-1}}{2}:\, \|u_\epsilon(t_k)-u(t_k)\|+\|u_\epsilon(s_k)-u(s_k)\|+\|\epsilon\dot{u}_{\epsilon}(s_k)\|+\|\epsilon\dot{u}_{\epsilon}(t_k)\|\leq\frac{1}{k}\right\},
\end{equation*}
which is well posed since, for every fixed $k$, it holds 
\begin{equation*}
\displaystyle\lim_{\epsilon\to0}\bigl[\|u_\epsilon(t_k)-u(t_k)\|+\|u_\epsilon(s_k)-u(s_k)\|+\|\epsilon\dot{u}_{\epsilon}(s_k)\|+\|\epsilon\dot{u}_{\epsilon}(t_k)\|\bigr]=0.
\end{equation*}
Along such subsequence, in view of \eqref{ipo1} and \eqref{ipo2}, we get \eqref{hyp1} and \eqref{hyp2}. As a consequence of {\bf (F0)} and \eqref{hyp1} we notice that
\begin{align}
&F_{s_k}(u_{\epsilon_k}(s_k))\to F_t(u_-(t)),\label{ipo3}\\
&F_{t_k}(u_{\epsilon_k}(t_k))\to F_t(u_+(t)),\label{ipo4}
\end{align}
as $k\to+\infty$.

The proof  distinguishes now between two cases:
\begin{itemize}
\item[]it holds either
\begin{align}\label{ilcasoA}
\displaystyle\mathop{\lim\sup}_k\frac{t_k-s_k}{\epsilon_k}=+\infty\,;
\end{align}
\item[] or
\begin{align}\label{ilcasoB}
\displaystyle\mathop{\lim\sup}_k\frac{t_k-s_k}{\epsilon_k}<+\infty\,.
\end{align}
\end{itemize}

{\it Proof if \eqref{ilcasoA} holds}. Let $\delta>0$ be fixed.  Define
\begin{equation}
\mathcal{C}_M(t):=\mathcal{C}(t)\cap B_M,
\label{CM}
\end{equation}
where $\mathcal{C}(t)$ is the set of critical points of $\nabla_xF_t(\cdot)$ as in \eqref{criticalset} and $M>0$ is such that $\|u_\epsilon(s)\|\leq M$, $s\in[0,T]$, as it follows by Proposition~\ref{apriori}(i). In the case $t=0$ we set 
\[
\mathcal{C}_M(0)=(\mathcal{C}(0)\cap B_M)\cup\{u_-(0)\}
\]
as the initial datum $u_-(0)$ might in general not be a critical point. In this case, we recall that, instead, the right limit $u_+(0)\in \mathcal{C}(0)$ by \eqref{(4.1)}. In both cases $t=0$ and $t \neq 0$, by assumption {\bf (F5)}, the set $\mathcal{C}_M(t)$ is finite, and we denote by $N_t$ its cardinality, $N_t:=\#(\mathcal{C}_M(t))<+\infty$. Since $\mathcal{C}_M(t)$ is finite, we may define the strictly positive value $d$ as
\begin{equation}
d=d_t:=\min\{\|w-z\|:\, w,z\in \mathcal{C}_M(t),\, w\neq z\}
\label{dt}
\end{equation}
and fix an arbitrary $\delta>0$ with the property that
\begin{equation}
\delta<\frac{d}{2}.
\label{smalldelta}
\end{equation}
With this,
\begin{equation*}
B_\delta(u^i)\cap B_\delta(u^j)=\emptyset,\quad \text{ for every } u^i,u^j\in\mathcal{C}_M(t),\quad i\neq j\,.
\end{equation*}
In particular, if for some $u\in B_M$, it holds
\begin{equation*}
\dist(u,\mathcal{C}_M(t))\leq\delta,
\end{equation*}
then there exists a unique $\bar{u}\in\mathcal{C}_M(t)$ such that
\begin{equation}
\dist(u,\mathcal{C}_M(t))=\|u-\bar{u}\|\leq\delta.
\label{uniqueproj}
\end{equation}
Since both $t_k\to t$ and $s_k\to t$, we may assume that $t_k-s_k\leq\delta$, for every $k\in\N$. Moreover, in view of the continuity of $\|\nabla_xF_{(\cdot)}(\cdot)\|$ ensured by {\bf(F0)}, we can fix $\eta>0$, with $\eta\leq\delta$, complying with the following property:
\begin{equation}
\text{if }(s,u)\in[s_k,t_k]\times B_M\text{ satisfies }\frac{1}{2}\|\nabla_x F_s(u)\|^2\leq\eta,\text{ then }\dist(u,\mathcal{C}_M(t))\leq\delta.
\label{(B)}
\end{equation}
We 
denote with $L$ a Lipschitz constant for $F_t(\cdot)$ on the set $B_M$.

We first prove the following\\
{\bf Claim:} for every $k\in\N$, there exists a finite collection of times
\begin{equation*}
s_k\le t_k^{1,-}<t_k^{1,+}\le \dots  \le t_k^{i,-}<t_k^{i,+}\le \dots \le  t_k^{{m_k},-}<t_k^{{m_k},+}\le t_k
\end{equation*}
with $m_k\leq N_t$, and a set of distinct critical points of $F_t(\cdot)$, say $\{u^1,\dots,u^{m_k}\}\subseteq \mathcal{C}(t)\cap B_M$, with  $u^{m_k}=u_+(t)$, such that, setting $u^0=u_-(t)$, the following properties are satisfied:
\begin{enumerate}
\item[(1)] $\|u_{\epsilon_k}(t^{1,-}_k)-u_{-}(t)\|=\delta$;
\item[(2)] $\|u_{\epsilon_k}(t^{m_k,+}_k)-u_{+}(t)\|\leq\delta$;
\item[(3a)] $\dist(u_{\epsilon_k}(t^{i,-}_k), \mathcal{C}_M(t))=\|u_{\epsilon_k}(t^{i,-}_k)-u^{i-1}\|=\delta$;
\item[(3b)] $\dist(u_{\epsilon_k}(t^{i,+}_k), \mathcal{C}_M(t))=\|u_{\epsilon_k}(t^{i,+}_k)-u^{i}\|\leq\delta$,\qquad \text{for every }$i=1,\dots, m_k$;
\item[(4)] $\displaystyle\mathop{\lim\sup}_{k\to +\infty}\frac{\epsilon_k^2}{2}\|\dot{u}_{\epsilon_k}(t_k^{i,-})\|_A^2\leq (2L+1)\delta\,,\quad\mathop{\lim\sup}_{k\to +\infty} \frac{\epsilon_k^2}{2}\|\dot{u}_{\epsilon_k}(t_k^{i,+})\|_A^2\leq \delta$;
\item[(5)] $\|u_{\epsilon_k}(t)-u^{i-1}\|>\delta$ for every $t>t^{i,-}_k$ and every $i=1,\dots,m_k$;
\item[(6)] $\displaystyle\mathop{\lim\sup}_{k\to +\infty}\frac{t^{i,+}_k-t^{i,-}_k}{\epsilon_k}\leq\frac{M}{\eta}<+\infty$.
\end{enumerate}
Notice that (3a), (3b) and (5) together imply that  $u^i\neq u^j$ for every $i,j$ with $i\neq j$. In order to prove the Claim, we will perform an algorithmic construction.\\ 
\\
{\bf Step 1.} Since $\|u_{\epsilon_k}(s_k)-u_-(t)\|\to0$ as $k\to+\infty$ and
\begin{equation*}
\displaystyle\mathop{\lim\inf}_{k\to\infty}\|{u}_{\epsilon_k}(t_k)-u_-(t)\|\geq\|u_+(t)-u_-(t)\|\geq d,
\end{equation*} 
it is well defined
\begin{equation}
t_k^{1,-}:=\max\{s\in[s_k,t_k]:\, \|u_{\epsilon_k}(s)-u_-(t)\|\leq\delta\},
\label{t-}
\end{equation}
and it satisfies $s_k\leq t_k^{1,-}<t_k$, $t_k^{1,-}\to t$ and then $t_k^{1,-}-s_k\to0$. Moreover
\begin{equation}
{\rm dist}({u}_{\epsilon_k}(t_k^{1,-}),\mathcal{C}_M(t))=\|{u}_{\epsilon_k}(t_k^{1,-})-u_-(t)\|=\delta.
\label{stimata}
\end{equation}
Observe that from {\bf (F0)}, \eqref{hyp1}, and \eqref{ipo3} we get
\begin{equation*}
\begin{split}
\displaystyle\mathop{\lim\sup}_{k\to+\infty} [F_{s_k}({u}_{\epsilon_k}(s_k))-F_{t_k^{1,-}}({u}_{\epsilon_k}(t_k^{1,-}))]&=\displaystyle\mathop{\lim\sup}_{k\to+\infty}[F_t(u_-(t))-F_t({u}_{\epsilon_k}(t_k^{1,-}))]\\
&\leq\displaystyle\mathop{\lim\sup}_{k\to+\infty}L\|{u}_{\epsilon_k}(t_k^{1,-})-u_-(t)\|\leq L\delta.
\end{split}
\end{equation*}
With this, rewriting the energy inequality \eqref{unostima} for $s=s_k$ and $t=t_k^{1,-}$, namely
\begin{equation*}
\begin{split}
\frac{\epsilon_k^2}{2}\|\dot{u}_{\epsilon_k}(t_k^{1,-})\|_A^2&\leq F_{s_k}({u}_{\epsilon_k}(s_k))-F_{t_k^{1,-}}({u}_{\epsilon_k}(t_k^{1,-}))+\frac{\epsilon_k^2}{2}\|\dot{u}_{\epsilon_k}(s_k)\|_A^2\\
&+\int_{s_k}^{t_k^{1,-}}\partial_r F_r(u_{\epsilon_k}(r))\,\mathrm{d}r,
\end{split}
\end{equation*}
and using \eqref{hyp2} we deduce
\begin{equation}
\displaystyle\mathop{\lim\sup}_{k\to+\infty}\frac{\epsilon_k^2}{2}\|\dot{u}_{\epsilon_k}(t_k^{1,-})\|_A^2\leq L\delta.
\label{(C)}
\end{equation}

We define a sequence of times $t^{1,+}_k>t^{1,-}_k$ as follows. From the bounds (iv) and (vii) of Proposition~\ref{apriori}, we have that 
\begin{equation}
\frac{1}{\epsilon_k}\int_{t^{1,-}_k}^{t^{1,-}_k+\frac{M}{\eta}\epsilon_k}\frac{1}{2}\|\nabla_x F_s(u_{\epsilon_k}(s))\|^2+\frac{\epsilon_k^2}{2}\|\dot{u}_{\epsilon_k}(s)\|^2\,\mathrm{d}s\leq M.
\label{(D)}
\end{equation}
If $t_k<t^{1,-}_k+\frac{M}{\eta}\epsilon_k$, we set $t^{1,+}_k:=t_k$. In this case, we deduce that
\begin{equation*}
\lim_{k\to+\infty}\left[\dist(u_{\epsilon_k}(t^{1,+}_k), \mathcal{C}_M(t))+\frac{\epsilon_k^2}{2}\|\dot{u}_{\epsilon_k}(t^{1,+}_k)\|^2\right]=0,
\end{equation*}
which in particular implies that
\begin{equation*}
\lim_{k\to+\infty}\left[\dist(u_{\epsilon_k}(t^{1,+}_k), \mathcal{C}_M(t))+\frac{\epsilon_k^2}{2}\|\dot{u}_{\epsilon_k}(t^{1,+}_k)\|^2\right]\leq\delta.
\end{equation*}
If instead $t_k>t^{1,-}_k+\frac{M}{\eta}\epsilon_k$, by \eqref{(D)} and the Mean Value Theorem there exists $t^{1,+}_k>t^{1,-}_k$ such that
\begin{equation*}
\frac{1}{2}\|\nabla_x F_{t^{1,+}_k}(u_{\epsilon_k}(t^{1,+}_k))\|^2+\frac{\epsilon_k^2}{2}\|\dot{u}_{\epsilon_k}(t^{1,+}_k)\|^2\leq\eta\,.
\end{equation*}
This implies
\begin{equation*}
\dist(u_{\epsilon_k}(t^{1,+}_k), \mathcal{C}_M(t))\leq\delta\,\,\text{ and }\,\,\frac{\epsilon_k^2}{2}\|\dot{u}_{\epsilon_k}(t^{1,+}_k)\|_A^2\leq\delta.
\end{equation*}
By \eqref{smalldelta}  and \eqref{uniqueproj} there exists a unique $u^1\in\mathcal{C}_M(t)$ such that
\begin{equation}
\dist(u_{\epsilon_k}(t^{1,+}_k), \mathcal{C}_M(t))=\|u_{\epsilon_k}(t^{1,+}_k)-u^1\|\leq\delta.
\label{(5.23)}
\end{equation}
Moreover, by construction, it holds
\begin{equation*}
\displaystyle\mathop{\lim\sup}_k\frac{t^{1,+}_k-t^{1,-}_k}{\epsilon_k}\leq \frac{M}{\eta}<+\infty,
\end{equation*}
while (5) is satisfied by \eqref{t-}. Now, if $u^1=u_+(t)$, the Claim is proved with $m_k=1$. Otherwise, the construction goes on.
\\
{\bf Step 2:} Assume that a collection 
\begin{equation*}
s_k\le t_k^{1,-}<t_k^{1,+}\le \dots  \le t_k^{j,-}<t_k^{j,+}\le t_k
\end{equation*}
has been constructed for $1\le j\le i$, such that all the properties in the Claim are satisfied with the exception of (2). By \eqref{hyp1} it then must be $u^{i}\neq u_+(t)$, $t_k^{j,+}\neq t_k$ and it is therefore well defined
\begin{equation}\label{t-new}
t_k^{i+1,-}:=\max\{s\in[s_k,t_k]:\, \|u_{\epsilon_k}(s)-u^{i}\|\leq\delta\}.
\end{equation}
By construction it holds $t_k^{i,+}\le t_k^{i+1,-}<t_k$; we have also
\begin{equation*}
\dist({u}_{\epsilon_k}(t_k^{i+1,-}),\mathcal{C}_M(t))=\|{u}_{\epsilon_k}(t_k^{i+1,-})-{u}^{i}\|=\delta.
\end{equation*}
Moreover, it holds the estimate
\begin{equation}
\displaystyle\mathop{\lim\sup}_{k\to+\infty}\frac{\epsilon_k^2}{2}\|\dot{u}_{\epsilon_k}(t_k^{i+1,-})\|_A^2\leq (2L+1)\delta\,.
\label{(Cbis)}
\end{equation}
This can be proved with an analogous argument as for \eqref{(C)}, replacing $s_k$ with $t_k^{i,+}$ and $t_k^{1,-}$ with $t_k^{i+1,-}$, using the bound provided by (4) for $\epsilon_k\|\dot u_{\epsilon_k}(t_k^{i,+})\|^2$. Observe indeed that $t_k^{i+1,-}-t_k^{i,+}\le t_k-s_k\to 0$ and that, by construction, $u_{\epsilon_k}(t_k^{i,+})$ and $u_{\epsilon_k}(t_k^{i+1,-})$ are close to the same point $u^i \in \mathcal{C}_M(t)$.

Now, as in the proof of Step 1, we have the bound
\begin{equation}
\frac{1}{\epsilon_k}\int_{t^{i+1,-}_k}^{t^{i+1,-}_k+\frac{M}{\eta}\epsilon_k}\frac{1}{2}\|\nabla_x F_s(u_{\epsilon_k}(s))\|^2+\frac{\epsilon_k^2}{2}\|\dot{u}_{\epsilon_k}(s)\|^2\,\mathrm{d}s\leq M.
\label{(Dbis)}
\end{equation} 
If $t_k<t^{i+1,-}_k+\frac{M}{\eta}\epsilon_k$, we set $t^{i+1,+}_k:=t_k$. Otherwise, by \eqref{(Dbis)} and the Mean Value Theorem there exists $t^{i+1,+}_k>t^{i+1,-}_k$ such that
\begin{equation*}
\frac{1}{2}\|\nabla_x F_{t^{i+1,+}_k}(u_{\epsilon_k}(t^{i+1,+}_k))\|^2+\frac{\epsilon_k^2}{2}\|\dot{u}_{\epsilon_k}(t^{i+1,+}_k)\|^2\leq\eta\,.
\end{equation*}
This implies
\begin{equation*}
\dist(u_{\epsilon_k}(t^{i+1,+}_k), \mathcal{C}_M(t))\leq\delta\,\,\text{ and }\,\,\frac{\epsilon_k^2}{2}\|\dot{u}_{\epsilon_k}(t^{i+1,+}_k)\|_A^2\leq\delta.
\end{equation*}
With \eqref{smalldelta}  and \eqref{uniqueproj} we again have that there exists a unique $u^{i+1}\in\mathcal{C}_M(t)$ such that
\begin{equation*}
\dist(u_{\epsilon_k}(t^{i+1,+}_k), \mathcal{C}_M(t))=\|u_{\epsilon_k}(t^{i+1,+}_k)-u^{i+1}\|\leq\delta.
\end{equation*}
Moreover, by construction, it holds
\begin{equation}\label{tempivicini}
\displaystyle\mathop{\lim\sup}_k\frac{t^{i+1,+}_k-t^{i+1,-}_k}{\epsilon_k}\leq \frac{M}{\eta}<+\infty,
\end{equation}
while (5) is satisfied by \eqref{t-}. Together with (3a) and (3b) this gives $u^{i+1}\neq u^{j}$ for all $j\le i$. With this, since $ \mathcal{C}_M(t)$ has a finite cardinality $N_t$ and recalling \eqref{hyp1}, in a finite number of steps $m_k\le N_t$ we will get property (2), concluding the proof of the Claim.

Let us go back to the main proof. Since $m_k\leq N_t$, up to passing to a subsequence, we may assume $m_k=m$ for any $k$, with $m$ independent of $k$. 
We also observe that, combining \eqref{mainequation} with \eqref{simpleid}, applied with $Q=B$, $z_1=\nabla_x F_r(u_{\epsilon_k}(r))+\epsilon_k^2A\ddot{u}_{\epsilon_k}(r)$, and $z_2=\epsilon_k \dot{u}_{\epsilon_k}(r)$ we get
\[
\begin{split}
0&=\|\nabla_x F_r(u_{\epsilon_k}(r))+\epsilon_k^2A\ddot{u}_{\epsilon_k}(r)\|_{B^{-1}}^2+2\langle \nabla_x F_r(u_{\epsilon_k}(r))+\epsilon_k^2A\ddot{u}_{\epsilon_k}(r), \epsilon_k \dot{u}_{\epsilon_k}(r)\rangle\\
&+\epsilon_k^2\|\dot{u}_{\epsilon_k}(r)\|_B^2
\end{split}
\]
for all $r \in [0, T]$.
With \eqref{hyp1} and \eqref{hyp2}, this gives
\begin{equation}\label{servedopo}
\begin{split}
&F_t(u_-(t))-F_t(u_+(t))\\
&=\displaystyle\lim_{k\to+\infty}\bigl[F_{s_k}(u_{\epsilon_k}(s_k))-F_{t_k}(u_{\epsilon_k}(t_k))
+ \frac{\epsilon_k^2}{2}\|\dot{u}_{\epsilon_k}(s_k)\|_A^2-\frac{\epsilon_k^2}{2}\|\dot{u}_{\epsilon_k}(t_k)\|_A^2\\
&+\int_{s_k}^{t_k}\partial_r F_r(u_{\epsilon_k}(r))\,\mathrm{d}r\bigr]\\
&=\displaystyle\lim_{k\to+\infty}\frac{1}{\epsilon_k}\int_{s_k}^{t_k}-\langle \nabla_x F_r(u_{\epsilon_k}(r))+\epsilon_k^2A\ddot{u}_{\epsilon_k}(r), \epsilon_k \dot{u}_{\epsilon_k}(r)\rangle\,\mathrm{d}r\\
&=\displaystyle\lim_{k\to+\infty}\frac{1}{2\epsilon_k}\int_{s_k}^{t_k}\|\nabla_x F_r(u_{\epsilon_k}(r))+\epsilon_k^2A\ddot{u}_{\epsilon_k}(r)\|_{B^{-1}}^2+\epsilon_k^2\|\dot{u}_{\epsilon_k}(r)\|_B^2\,\mathrm{d}r\\
&\geq \displaystyle\mathop{\lim\inf}_{k\to+\infty}\sum_{i=1}^m\frac{1}{2\epsilon_k}\int_{t_k^{i,-}}^{t_k^{i,+}}\|\nabla_x F_r(u_{\epsilon_k}(r))+\epsilon_k^2A\ddot{u}_{\epsilon_k}(r)\|_{B^{-1}}^2+\epsilon_k^2\|\dot{u}_{\epsilon_k}(r)\|_B^2\,\mathrm{d}r.
\end{split}
\end{equation}
Let $\omega$ be a modulus of continuity for $\nabla_x F_{(\cdot)}(\cdot)$ on $[0,T]\times B_M$. By assumption {\bf(F0)}, the $L^\infty$-bounds of Proposition \ref{apriori} (i) and (iii), and \eqref{tempivicini} we note that, for each fixed $i=1,\dots,m$, it results
\begin{equation*}
\begin{split}
\frac{1}{2\epsilon_k}&\Bigl|\int_{t_k^{i,-}}^{t_k^{i,+}}\|\nabla_x F_r(u_{\epsilon_k}(r))+\epsilon_k^2A\ddot{u}_{\epsilon_k}(r)\|_{B^{-1}}^2-\|\nabla_x F_t(u_{\epsilon_k}(r))+\epsilon_k^2A\ddot{u}_{\epsilon_k}(r)\|_{B^{-1}}^2\,\mathrm{d}r\Bigr|\\
&\leq \frac{1}{2\epsilon_k}\int_{t_k^{i,-}}^{t_k^{i,+}}\left|\|\nabla_x F_r(u_{\epsilon_k}(r))\|_{B^{-1}}^2-\|\nabla_x F_t(u_{\epsilon_k}(r))\|_{B^{-1}}^2\,\right|\mathrm{d}r\\
&+\frac{\|B^{-1}\|}{\epsilon_k}\int_{t_k^{i,-}}^{t_k^{i,+}}\|\nabla_x F_r(u_{\epsilon_k}(r))-\nabla_x F_t(u_{\epsilon_k}(r))\|\|\epsilon_k^2A\ddot{u}_{\epsilon_k}(r)\|\,\mathrm{d}r\\
&\leq C\frac{(t_k^{i,+}-t_k^{i,-})}{\epsilon_k}\left[\max\{|t_k^{i,+}-t|,|t_k^{i,-}-t|\}+\omega(\max\{|t_k^{i,+}-t|,|t_k^{i,-}-t|\})\right]\\
&\leq C\frac{M}{\eta}\left[\max\{|t_k^{i,+}-t|,|t_k^{i,-}-t|\}+\omega(\max\{|t_k^{i,+}-t|,|t_k^{i,-}-t|\})\right]\to0
\end{split}
\end{equation*}
as $k\to+\infty$. Thus, we get
\begin{equation}
\begin{split}
&F_t(u_-(t))-F_t(u_+(t))\\
&\geq \displaystyle\mathop{\lim\inf}_{k\to+\infty}\sum_{i=1}^m\frac{1}{2\epsilon_k}\int_{t_k^{i,-}}^{t_k^{i,+}}\|\nabla_x F_t(u_{\epsilon_k}(r))+\epsilon_k^2A\ddot{u}_{\epsilon_k}(r)\|_{B^{-1}}^2+\epsilon_k^2\|\dot{u}_{\epsilon_k}(r)\|_B^2\,\mathrm{d}r.
\end{split}
\end{equation}
With fixed $i\in\{1,\dots,m\}$, we set
\begin{equation*}
v_k(\tau):=u_{\epsilon_k}(\epsilon_k\tau+t_k^{i,-}),\quad \text{ for every }\tau\in\R.
\end{equation*}
From the $L^\infty$ bounds of Proposition \ref{apriori} (i)-(iii) we immediately deduce the equi-boundedness of $v_k$ in $W^{2,\infty}(\R)$. Moreover, through the change of variables $r=\epsilon_k\tau+t_k^{i,-}$ we obtain
\begin{equation*}
\begin{split}
&\displaystyle\frac{1}{2\epsilon_k}\int_{t_k^{i,-}}^{t_k^{i,+}}\|\nabla_x F_t(u_{\epsilon_k}(r))+\epsilon_k^2A\ddot{u}_{\epsilon_k}(r)\|_{B^{-1}}^2+\epsilon_k^2\|\dot{u}_{\epsilon_k}(r)\|_B^2\,\mathrm{d}r\\
&=\displaystyle\frac{1}{2}\int_{0}^{\frac{t_k^{i,+}-t_k^{i,-}}{\epsilon_k}}\|\nabla_x F_t(v_{k}(\tau))+A\ddot{v}_{k}(\tau)\|_{B^{-1}}^2+\|\dot{v}_{k}(\tau)\|_{B}^2\,\mathrm{d}\tau.
\end{split}
\end{equation*}
In order to simplify notation, here and in the following we will denote by $\sigma^{i}_k$ the ratio $\frac{t_k^{i,+}-t_k^{i,-}}{\epsilon_k}$. The bounds (1)-(4) can be re-read for $v_k$ as
\begin{equation}\label{rewrite}
\begin{split}
\|v_k(0)-u^{i-1}\|\leq\delta, \quad \|v_k(\sigma^{i}_k)-u^i\|\leq\delta,\\
\|\dot{v}_k(0)\|_A^2\leq 2(2L+1)\delta, \quad \|\dot{v}_k(\sigma^{i}_k)\|_A^2\leq2\delta\,.
\end{split}
\end{equation}
Consider the functions
\begin{equation*}
g(p)=3p^2-2p^3,\quad h(p)=-p(1-p),\quad p\in[0,1],
\end{equation*} and the competitor
\begin{equation*}
\tilde{v}_k(\tau)=
\begin{cases}
u^{i-1}, & \tau\leq-1,\\
u^{i-1}+g(\tau+1)(v_k(0)-u^{i-1})+h(\tau+1)\dot{v}_k(0), & \tau\in[-1,0],\\
v_k(\tau), & \tau\in[0,\sigma^i_k],\\
v_k(\sigma^i_k)+g(\tau-\sigma^i_k) (u^i-v_k(\sigma^i_k))-h(\tau-\sigma^i_k)\dot{v}_k(\sigma^i_k), & \tau\in[\sigma^i_k,\sigma^i_k+1],\\
u^i, & \tau\geq \sigma^i_k+1.
\end{cases}
\end{equation*}
Fix an arbitrary $N \in \mathbb{N}$ with $2N+1>\sigma^i_k+1$, $\tilde{v}_k \in V^{t, N}_{u^{i-1}, u^i}$. Since $u^i \in \mathcal{C}(t)$ for all $i\ge 1$ (notice that this holds also for $t=0$), exploiting \eqref{rewrite} we  obtain that there exists a uniform constant $C$ such that
\begin{equation}
\begin{split}
&\frac{1}{2}\int_{-1}^{2N+1}\|\nabla_xF_t(\tilde{v}_k(\tau))+A\ddot{\tilde{v}}_k(\tau)\|_{B^{-1}}^2+\|\dot{\tilde{v}}_k(\tau)\|_B^2\,\mathrm{d}\tau\\
&\leq \frac{1}{2}\int_{0}^{\sigma^i_k}\|\nabla_xF_t(v_k(\tau))+A\ddot{v}_k(\tau)\|_{B^{-1}}^2+\|\dot{v}_k(\tau)\|_B^2\,\mathrm{d}\tau\\
&+2 [C\delta+\omega^2(C\sqrt{\delta})],
\end{split}
\label{(c*)}
\end{equation}
where $\omega$ is a modulus of continuity for $\nabla_xF_t(\cdot)$ on $B_M$. Above we additionally exploited \eqref{smalldelta} to estimate $\delta+\sqrt{\delta}$ only in terms of $\sqrt{\delta}$ inside the function $\omega^2$.  

With the time translation $\hat v_k(s)=\tilde{v}_k(\tau+N-1)$ we get a competitor for $c_t(u^{i-1}, u^{i})$ with
\[
\begin{split}
\int_{-1}^{2N+1}\|\nabla_xF_t(\tilde{v}_k(\tau))+A\ddot{\tilde{v}}_k(\tau)\|_{B^{-1}}^2+\|\dot{\tilde{v}}_k(\tau)\|_B^2\,\mathrm{d}\tau=\\
\int_{-N}^{N}\|\nabla_xF_t(\hat{v}_k(s))+A\ddot{\hat{v}}_k(s)\|_{B^{-1}}^2+\|\dot{\hat{v}}_k(s)\|_B^2\,\mathrm{d}s\,.
\end{split}
\]
Then, from \eqref{(c*)} and condition (4) of Theorem~\ref{cost}, we get
\begin{equation*}
\begin{split}
c_t(u^0,u^m)&\leq \sum_{i=1}^m c_t(u^{i-1},u^i)\\
&\leq 2m[C\delta+\omega^2(C\sqrt{\delta})]+\sum_{i=1}^m\frac{1}{2}\int_{0}^{\sigma^i_k}\|\nabla_xF_t(v_k(\tau))+A\ddot{v}_k(\tau)\|_{B^{-1}}^2+\|\dot{v}_k(\tau)\|_B^2\,\mathrm{d}\tau,
\end{split}
\end{equation*}
whence, passing to the limit as $k\to+\infty$, we finally obtain
\begin{equation*}
\begin{split}
c_t(u^0,u^m)&\leq 2m[C\delta+\omega^2(C\sqrt{\delta})]+ F_t(u_-(t))-F_t(u_+(t))\\
& \leq 2N_t[C\delta+\omega^2(C\sqrt{\delta})]+ F_t(u_-(t))-F_t(u_+(t)),
\end{split}
\end{equation*}
from which, by the arbitrariness of $\delta>0$ and taking into account the symmetry of the cost function $c_t(u_-(t),u_+(t))=c_t(u_+(t),u_-(t))$, we deduce 
\begin{equation*}
c_t(u_-(t),u_+(t))\leq F_t(u_-(t))-F_t(u_+(t)),
\end{equation*}
since by (1) and (2) of Claim it must be $u^0=u_-(t)$ and $u^m=u_+(t)$. This concludes the proof, if \eqref{ilcasoA} holds.

{\it Proof if \eqref{ilcasoB} holds}. In this case, one does not need to prove the Claim, since trajectories already move from $u_-(t)$ to $u_+(t)$ in a time interval whose  length is a $O(\epsilon_k)$. One can then perform an analogous argument as in the previous case with $m=1$, $t^{1,-}_k=s_k$ and $t^{1,+}_k=t_k$, starting from the chain of inequalities in \eqref{servedopo}. Observe indeed that \eqref{ilcasoB} replaces \eqref{tempivicini}, while the conditions in \eqref{rewrite}, with $i=1$, are in this case automatically satisfied because of \eqref{hyp1}, and \eqref{hyp2}.
\endproof

\begin{oss}\label{equality}
As a consequence of Theorem~\ref{cost}(3) and Proposition~\ref{proposizione1}, we get the equality
\begin{equation*}
F_t(u_-(t))-F_t(u_+(t))=c_t(u_+(t),u_-(t)),\quad \forall t\in[0,T].
\end{equation*}
\end{oss}

We conclude with the following theorem that summarizes the results of this section and characterizes $u$ as a Balanced Viscosity solution of the problem
\begin{equation*}
\nabla_x F_t(u(t))=0\,\text{ in $X$}\quad\text{ for a.e. $t\in[0,T]$.}
\end{equation*}

\begin{thm}\label{main1}
Assume that {\bf (F0)-(F6)} hold, with {\bf(F3')} in place of {\bf(F3)}, and let ${u}_{\epsilon}: [0,T]\longrightarrow X$ be the solution of the Cauchy problem associated to \eqref{mainequation} with initial condition at $t=0$ and $u_\epsilon(0)$ be uniformly bounded, $\epsilon \dot{u}_\epsilon(0)\to0$, as $\epsilon\to0$. Let $c_t$ be the cost function defined in {\rm(\ref{costfun})}. Then, up to a subsequence independent of $t$, $(u_\epsilon)_\epsilon$ converge pointwise, as $\epsilon\to0$, to a function $u:[0,T]\longrightarrow X$. Moreover, $u$ satisfies the following properties:
\begin{enumerate}
\item[\rm(i)] $u$ is regulated;
\item[\rm(ii)] it holds
\begin{equation*}
\nabla_x F_t(u_+(t))=\nabla_xF_t(u_-(t))=0 \quad \text{in $X$ for every $t\in (0,T]$};
\end{equation*}
\item[\rm(iii)] $u$ fulfills the energy balance 
\begin{equation}
\begin{split}
F_t(u_+(t)) &+ \sum_{r\in J\cap[s,t]} c_r(u_-(r),u_+(r))=F_s(u_-(s))\\
    &+\int_s^t \partial_r F_r(u(r))\,dr,\quad \text{ for every }\,0\leq s\leq t\leq T.
\end{split}
\label{ebalance}
\end{equation}
\end{enumerate}
\label{bvthm}
\end{thm}

\subsection{Behavior at jump points}

As a final result, we want to characterize the behavior of the limit evolution at the jump times, by showing, with Theorem~\ref{variat}, that the left and right limits $u_-(t)$ and $u_+(t)$, respectively, are connected by a finite number of heteroclinic solutions to the unscaled autonomous equation
\begin{equation*}
A\ddot{v}(s)+B\dot{v}(s)+\nabla_xF_t(v(s))=0.
\end{equation*}
We  will mainly focus on the case where both $u_-(t)$ and $u_+(t)$ are stationary points of $F(t, \cdot)$. The only exception could happen for $t=0$, where it might be $u_-(0)\notin \mathcal{C}(0)$: this case only requires minor adaptions, and, differently than in the proof of Proposition \ref{proposizione1}, we will not give details, but only provide the slightly  different statement in Remark \ref{fine}.

We start with a simple lemma dealing with the asymptotic behavior of functions $v$ such that $\dot v(s)$ and $\nabla_x F(t,v(s))\in L^2(\R)$. Namely, we prove that both their limits at infinity $v(-\infty)$ and $v(+\infty)$ exist and belong to the set of critical points of $F(t,\cdot)$.

\begin{lem}\label{Lemma1}
Assume {\bf (F5)}. Let $t\in[0,T]$ be fixed and $v\in W^{1,2}(\R;X)\cap L^\infty(\R;X)$ be such that
\begin{equation}
\int_{-\infty}^{+\infty}\|\nabla_x F_t(v(s))\|^2\,\mathrm{d}s<+\infty.
\label{condition}
\end{equation}
Then the following limits exist finite:
\begin{align}
\lim_{s\to-\infty}v(s)=v_*\,,\label{limits1}\\
\lim_{s\to+\infty}v(s)=v^*\,,\label{limits2}
\end{align}
and $v_*, v^*\in\mathcal{C}(t)$. If, in addition, $v\in W^{2,2}(\R;X)$ then it also holds
\begin{equation}
\lim_{s\to\pm\infty}\dot{v}(s)=0.
\label{limitsder}
\end{equation}
\end{lem}

\proof
We only prove \eqref{limits1}, since the proof of \eqref{limits2} is similar. Define the limit class of $v$ as
\begin{equation*}
\omega:=\left\{w\in\R\cup\{\pm\infty\}:\, \exists\, s_k\to-\infty\,\text{ such that }\, \lim_{k\to+\infty}v(s_k)=w\right\}.
\end{equation*}
Clearly, $\omega$ is nonempty and $\omega\subseteq\R$ since $v$ is bounded, say $\|v\|_{L^\infty(\R)}\leq M$ for some $M>0$. In order to prove \eqref{limits1}, we have to show that $\omega=\{v_*\}$ for some $v_*\in\R$. We then argue by contradiction and assume that $\omega$ contains at least two  points $w_1$ and $w_2$, with $w_1\neq w_2$. Let $s_k,t_k\to-\infty$ be such that $v(s_k)\to w_1$ and $v(t_k)\to w_2$. Up to a further extraction, we may assume that $s_k\leq t_k$, $t_{k}\leq s_{k+1}$ for every $k\in\N$ and $s_k\leq t_k<-\eta$ for every $k$ large enough. Now, since $\mathcal{C}_M(t):=\mathcal{C}(t)\cap B_M$ is finite, the same argument as in the proof of Proposition~\ref{propdelta} provides the existence of $\delta=\delta(t,w_1,w_2)>0$ and $k_0\in\N$ such that
\begin{equation*}
\delta<\int_{s_k}^{t_k}\|\nabla_xF_t(v(s))\|\|\dot{v}(s)\|\mathrm{d}s
\end{equation*}
for every $k\geq k_0$. Moreover, by applying the Cauchy-Schwarz inequality, we get
\begin{equation*}
\delta<\frac{1}{2}\int_{s_k}^{t_k}\|\nabla_xF_t(v(s))\|^2+\|\dot{v}(s)\|^2\mathrm{d}s,
\end{equation*}
whence, summing up $k$, we deduce that
\begin{equation*}
\sum_{k=k_0}^{+\infty}\delta<\frac{1}{2}\sum_{k=k_0}^{+\infty}\int_{s_k}^{t_k}\|\nabla_xF_t(v(s))\|^2+\|\dot{v}(s)\|^2\mathrm{d}s<\frac{1}{2}\int_{-\infty}^{-\eta}\|\nabla_xF_t(v(s))\|^2+\|\dot{v}(s)\|^2\mathrm{d}s,
\end{equation*}
that contradicts \eqref{condition} and the fact that $\dot{v}\in L^2(\R)$. Thus, there exists $v_*\in\R$ such that $\omega=\{v_*\}$ and \eqref{limits1} holds. Moreover, it must be $v_*\in\mathcal{C}(t)$ in order to have
\begin{equation*}
\int_{-\infty}^{\eta}\|\nabla_x F_t(v(s))\|^2\,\mathrm{d}s<+\infty,
\end{equation*}
for every fixed $\eta>0$.

Assume now, in addition, that $\ddot{v}\in L^2(\R)$. We can choose a sequence $t_k\to-\infty$ such that $\dot{v}(t_k)\to0$ as $k\to+\infty$. By a simple computation and the Cauchy-Schwarz inequality we have
\begin{equation}
\begin{split}
\frac{1}{2}\|\dot{v}(t)\|^2&=\frac{1}{2}\|\dot{v}(t_k)\|^2+\int_{t_k}^t\langle\ddot{v}(s),\dot{v}(s)\rangle\,\mathrm{d}s\\
  & \leq \frac{1}{2}\|\dot{v}(t_k)\|^2+\frac{1}{2}\int_{-\infty}^t\|\ddot{v}(s)\|^2+\|\dot{v}(s)\|^2\,\mathrm{d}s.
\end{split}
\label{previous}
\end{equation} 
For every fixed $\epsilon$, there exists $\bar{t}<0$ such that, for every $t\leq\bar{t}$, the integral at the right hand side of \eqref{previous} is smaller than $\epsilon$, thus obtaining
\begin{equation*}
\frac{1}{2}\|\dot{v}(t)\|^2\leq \frac{1}{2}\|\dot{v}(t_k)\|^2+\epsilon,\quad \forall\, t\leq\bar{t},
\end{equation*}
whence, letting $t_k\to-\infty$, we finally deduce $\displaystyle\lim_{t\to-\infty}\dot{v}(t)=0$, as desired.
\endproof

We will also make use of the following technical Lemma.

\begin{lem}\label{Lemma2}
Let $(f_k)_{k\in\N}$ be a sequence of functions such that $f_k\to f$ as $k\to+\infty$ uniformly on the compact subsets of $\R$, and assume that
\begin{equation}
\lim_{x\to+\infty} f(x)=f^*.
\label{condilim}
\end{equation}  
Then there exist $x_j\to+\infty$ and $(k_j)_{j\in\N}$ such that
\begin{equation}
\|f_{k_j}(x_j+\tau)-f^*\|\to0
\label{thes}
\end{equation}
as $j\to+\infty$, for every $\tau\in\R$.
\end{lem}

\proof

We argue by induction on $j\geq1$. By assumption \eqref{condilim}, we can fix $\bar{x}_j>\bar{x}_{j-1}$ such that
\begin{equation*}
\|f(x)-f^*\|\leq\frac{1}{2^j},\quad \text{ for every }x\geq \bar{x}_j,
\end{equation*}  
and, with the local uniform convergence $f_k\to f$, we can choose $k_j>k_{j-1}$ such that
\begin{equation*}
\|f_{k_j}(x)-f(x)\|\leq\frac{1}{2^j},\quad \text{ for every }x\in [\bar{x}_j,\bar{x}_j+j].
\end{equation*}
Setting $x_j:=\bar{x}_j+\frac{j}{2}$, for any fixed $\tau\in\R$ we have $x_j+\tau\in[\bar{x}_j,\bar{x}_j+j]$ for every $j\geq 2|\tau|$, and then
\begin{equation*}
\|f_{k_j}(x_j+\tau)-f^*\|\leq \sup_{x\geq \bar{x}_j}\|f(x)-f^*\|+\sup_{x\in [\bar{x}_j,\bar{x}_j+j]}\|f_{k_j}(x)-f(x)\|\leq \frac{1}{2^{j-1}}
\end{equation*}
from which, passing to the limit as $j\to+\infty$, we get \eqref{thes}.

\endproof

Now, we can state and prove the announced result. Observe that condition \eqref{bal} in the statement  is indeed satisfied for $u=u_-(t)$ and $v=u_+(t)$ for any $t \in J$ , as pointed out in Remark \ref{equality}. For $t=0$, the assumption $u_-(0)$ to be a critical point of $F(0,\cdot)$ might be not verified: this case requires a minor modification and will be shortly discussed in Remark \ref{fine}.

\begin{thm}\label{variat}
Assume {\bf(F0)-(F5)}. Let $t\in[0,T]$ be fixed and assume that there exist $u,v\in\mathcal{C}(t)$ such that
\begin{equation}
F_t(u)-F_t(v)=c_t(u,v),
\label{bal}
\end{equation}
where $c_t$ is the cost function defined by \eqref{costfun}. Then, there exist a subset of distinct points $\{u^0,u^1,\dots,u^m\}\subseteq\mathcal{C}(t)$ with $u^0=u$, $u^m=v$ and a family of functions $(v^i)_{i=1,\dots,m}$ such that, for every $i=1,\dots,m$,
\begin{equation}
\begin{cases}
&A\ddot{v}^i(s)+B\dot{v}^i(s)+\nabla_xF_t(v^i(s))=0,\quad\forall\,s\in\R;\\
&\displaystyle\lim_{s\to-\infty}v^i(s)=u^{i-1}, \displaystyle\lim_{s\to+\infty}v^i(s)=u^{i};\\
&\displaystyle\lim_{s\to\pm\infty}\dot{v}^i(s)=0.
\end{cases}
\label{properties}
\end{equation}
\end{thm}

\proof
The argument is based on a recursive construction  on the number $i\geq1$ of functions $v^i$ complying with \eqref{properties}. We subdivide it into two steps.\\
{\bf Step1.} First, we fix $u_k$ to be an infimizing sequence for the infimum problem \eqref{costfun+}, which equivalently defines $c_t(u,v)$ according to (5) in Theorem \ref{cost}; i.e., a sequence $u_k\in W^{2,2}(\R;X)$ such that $u_k(-\infty)=u$, $u_k(+\infty)=v$ and
\begin{equation}
\displaystyle\lim_{k\to+\infty}\frac{1}{2}\int_{-\infty}^{+\infty}\|\nabla_xF_t(u_k(s))+A\ddot{u}_k(s)\|_{B^{-1}}^2+\|\dot{u}_k(s)\|_B^2\,\mathrm{d}s = c_t(u,v).
\label{minimizing}
\end{equation}
With \eqref{bal} we then have
\begin{equation}
\displaystyle\lim_{k\to+\infty}\frac{1}{2}\int_{-\infty}^{+\infty}\|\nabla_xF_t(u_k(s))+A\ddot{u}_k(s)\|_{B^{-1}}^2+\|\dot{u}_k(s)\|_{B}^2\,\mathrm{d}s = F_t(u)-F_t(v).
\label{6.17}
\end{equation}
Furthermore,  since by Lemma \ref{Lemma1} it results $\dot{u}_k(\pm \infty)=0$ for every fixed $k\in\N$, we get that
\begin{equation*}
\begin{split}
\int_{-\infty}^{+\infty}\langle \nabla_xF_t(u_k(s))+A\ddot{u}_k(s), \dot{u}_k(s)\rangle\,\mathrm{d}s&=F_t(u_k(r))+\frac{1}{2}\|\dot{u}_k(r)\|_A^2\Big|_{-\infty}^{+\infty}\\
&= F_t(u_k(+\infty))-F_t(u_k(-\infty))=F_t(v)-F_t(u),
\end{split}
\end{equation*}
Summing up with \eqref{6.17}, and eventually using \eqref{simpleid}, we deduce
\begin{equation}
\begin{split}
\displaystyle0=\lim_{k\to+\infty}&\int_{-\infty}^{+\infty}\frac{1}{2}\left\{\|\nabla_xF_t(u_k(s))+A\ddot{u}_k(s)\|_{B^{-1}}^2+\|\dot{u}_k(s)\|_{B}^2\right\}\,\mathrm{d}s\\
+&\int_{-\infty}^{+\infty}\langle \nabla_xF_t(u_k(s))+A\ddot{u}_k(s), \dot{u}_k(s)\rangle\,\mathrm{d}s\\
=\displaystyle\lim_{k\to+\infty}&\int_{-\infty}^{+\infty}\frac{1}{2}\|A\ddot{u}_k(s)+B\dot{u}_k(s)+ \nabla_xF_t(u_k(s))\|_{B^{-1}}^2\,\mathrm{d}s\,.
\end{split}
\label{(star)}
\end{equation}

Notice that \eqref{6.17} provides an upper bound on 
\[
\frac{1}{2}\int_{-\infty}^{+\infty}\|\nabla_xF_t(u_k(s))+A\ddot{u}_k(s)\|_{B^{-1}}^2+\|\dot{u}_k(s)\|_{B}^2\,\mathrm{d}s
\]
which is independent of $k$. With this,  \eqref{equa1n}-\eqref{equa3} hold for a constant $M$ not depending on $k$, 
and we deduce that the sequence $(u_k(s))_k$ is equi-bounded in $W^{2,2}(\R;X)\cap L^\infty(\R;X)$.
Thus, in particular, there exists $M=M(t,u,v)>0$ such that
\begin{equation*}
\|u_k(s)\|\leq M,\quad \text{ for every }s\in\R\text{ and }k\in\N.
\end{equation*}
Correspondingly, we define the set $\mathcal{C}_M(t):=\mathcal{C}(t)\cap B_M$ and $d_t$ as in \eqref{CM} and \eqref{dt}, respectively, and we fix $d \le \frac14 d_t$.

Now we set $u^0:=u$ and, for every $k\in\N$,
\begin{equation}
t_k^1:=\min\{s\in\R:\,\|u_k(s)-u^0\|=d\},
\label{defn1}
\end{equation}
where the minimum is well posed since $u_k(-\infty)=u$ and $u_k(+\infty)=v\neq u$. 

Then, we consider the time-translations by $t_k^1$ of $u_k$, namely 
\begin{equation*}
v^1_k(s):=u_k(t^1_k+s),\quad\forall s\in\R.
\end{equation*}
We notice that, by the definition of $t^1_k$, it holds
\begin{equation}
\|v^1_k(s)-u^0\|\leq d\le \frac14 d_t,\quad\mbox{for every } s\leq0,
\label{condA}
\end{equation}
while for every $s$, they comply with the identity
\begin{equation}
F_t(v^1_k(s))+\frac{1}{2}\|\dot{v}_k(s)\|_A^2=F_t(u)+\int_{-\infty}^s\langle\nabla_xF_t(v^1_k(r))+A\ddot{v}^1_k(r),\dot{v}^1_k(r)\rangle\,\mathrm{d}r.
\label{star0}
\end{equation}
Now, by \eqref{(star)}, we obtain
\begin{equation}
\begin{split}
&\displaystyle\mathop{\lim\inf}_{k\to+\infty}\int_{-\infty}^s\langle\nabla_xF_t(v^1_k(r))+A\ddot{v}^1_k(r),\dot{v}^1_k(r)\rangle\,\mathrm{d}r\\
&=\displaystyle\mathop{\lim\inf}_{k\to+\infty}\left(-\frac{1}{2}\int_{-\infty}^{s}\|\nabla_xF_t(v^1_k(r))+A\ddot{v}^1_k(r)\|_{B^{-1}}^2+\|\dot{v}^1_k(r)\|_B^2\,\mathrm{d}r\right).
\end{split}
\label{(star2)}
\end{equation}
Moreover, for every fixed $k$, since $\|v^1_k(0)-u^0\|=d$ and $v^1_k(-\infty)=u^0$, by arguing as in the proof of \eqref{newdelta} we can find $\gamma_1=\gamma_1(t,u^0,d,M,\beta)>0$ such that
\begin{equation*}
\frac{1}{2}\int_{-\infty}^{0}\|\nabla_xF_t(v^1_k(r))+A\ddot{v}^1_k(r)\|_{B^{-1}}^2+\|\dot{v}^1_k(r)\|_B^2\,\mathrm{d}r>\gamma_1,
\end{equation*}
and this obviously implies 
\begin{equation}
\frac{1}{2}\int_{-\infty}^{s}\|\nabla_xF_t(v^1_k(r))+A\ddot{v}^1_k(r)\|_{B^{-1}}^2+\|\dot{v}^1_k(r)\|_B^2\,\mathrm{d}r>\gamma_1,
\label{conddelta}
\end{equation}
for every $s\geq0$.
Thus, combining \eqref{star0}, \eqref{(star2)} and \eqref{conddelta} we deduce
\begin{equation}
\displaystyle\mathop{\lim\inf}_{k\to+\infty}\left(F_t(v^1_k(s))+\frac{1}{2}\|\dot{v}^1_k(s)\|_A^2\right)\leq F_t(u^0)-\gamma_1,\quad\text{ for every }s\geq0.
\label{condB}
\end{equation}

The sequence $v^1_k$ is equi-bounded in $W^{2,2}(\R;X)$, since it is a time-translation of $u_k$,  for which \eqref{equa1n}-\eqref{equa3} hold. By virtue of Ascoli-Arzel\`a Theorem there exists a (not relabeled) subsequence of $v_k^1$ and a function $v^1$ such that $v^1_k\to v^1$, $\dot{v}^1_k\to\dot{v}^1$ uniformly on the compact subsets of $\R$, and $v^1_k\to v^1$ weakly in $W^{2,2}(\R;X)$.
By semicontinuity and from \eqref{(star)}, for every arbitrary closed interval $[a,b]\subset\R$, we obtain
\begin{equation*}
\frac{1}{2}\int_{a}^{b}\|A\ddot{v}^1(r)+B\dot{v}^1(r)+\nabla_xF_t(v^1(r))\|_{B^{-1}}^2\,\mathrm{d}r\leq 0
\end{equation*}
whence we finally get
\begin{equation*}
A\ddot{v}^1(r)+B\dot{v}^1(r)+\nabla_xF_t(v^1(r))=0
\end{equation*}
for every $r \in \mathbb{R}$
Similarly, we also deduce
\begin{equation*}
\int_{-\infty}^{+\infty}\|\nabla_xF_t(v^1(r))\|^2\,\mathrm{d}r\leq M
\end{equation*}
where $M$ is given by \eqref{equa3}.
Therefore, by virtue of Lemma~\ref{Lemma1}, there exist the limits $\displaystyle\lim_{r\to-\infty}v^1(r)=v^1_*\in\mathcal{C}_M(t)$, $\displaystyle\lim_{r\to+\infty}v^1(r)=v^{1,*}\in\mathcal{C}_M(t)$ and $\displaystyle\lim_{r\to\pm\infty}\dot{v}^1(r)=0$. Thus, from \eqref{condA}, it must be $v^1_*=u^0$. From \eqref{condB} it follows that
\begin{equation}
\left(F_t(v^1(r))+\frac{1}{2}\|\dot{v}^1(r)\|_A^2\right)\leq F_t(u^0)-\gamma_1,\quad\text{ for every }r\geq0,
\label{condBlim}
\end{equation}
whence, in particular, 
\begin{equation*}
F_t(v^{1,*})<F_t(u^0)
\end{equation*}
and then $v^{1,*}\neq u^0$. Now, define $u^1:=v^{1,*}\in\mathcal{C}_M(t)$. If $u^1=v$, then $m=1$ and construction stops here, otherwise the proof goes on as follows.\\
{\bf Step 2.} Let $i\geq2$. Assume that the sequences $(v^{l}_k(r))_k$ have been constructed for every $1\leq l\leq i-1$ and the corresponding limits $v^l$ comply with \eqref{properties}, for some $u^{l-1}, u^{l}\in\mathcal{C}_M(t)$.

In order to define $v^i$, we apply Lemma~\ref{Lemma2} to the sequence $(v^{i-1}_k(r),\dot{v}^{i-1}_k(r))$. Therefore, there exist a sequence $t_j\to+\infty$ and a subsequence $k_j\to+\infty$ such that 
\begin{equation}
(v^{i-1}_{k_j}(t_j+s),\dot{v}^{i-1}_{k_j}(t_j+s))\to(u^{i-1},0), \text{ as } j\to+\infty, \text{ for every } s\in\R. 
\label{condC}
\end{equation}
Now we define
\begin{equation*}
t_j^i:=\min\{s\geq t_j:\, \|v^{i-1}_{k_j}(s)-u^{i-1}\|=d\},
\end{equation*}
and we note that this definition is well posed, since $v^{i-1}_{k_j}$ are translations in time of $u_{k_j}$ and $u_{k_j}(+\infty)=v\neq u^{i-1}$. 
Moreover, $t_j^i-t_j\to+\infty$ by \eqref{condC}. Setting
\begin{equation*}
v^i_j(s):=v^{i-1}_{k_j}(t^i_j+s),\quad s\in\R,
\end{equation*}
for every $s\leq0$, there exists $j=j(s)$ such that
\begin{equation}
\|v^i_j(\tau)-u^{i-1}\|\leq d,\quad \forall \tau\in[s,0],\, \forall j\geq j(s)
\label{boundbis}
\end{equation}
(namely, it will suffice to choose $t_j^i-t_j\geq|s|$). 
Being  translations in time of $u_{k_j}$, with \eqref{(star)} we can show that $v^i_j$ comply with
\begin{equation}
\lim_{j\to+\infty}\int_{-\infty}^{+\infty}\frac{1}{2}\|A\ddot{v}^i_j(s)+B\dot{v}^i_j(s)+\nabla_xF_t(v^i_j(s))\|_{B^{-1}}^2\,\mathrm{d}s=0
\label{condF}
\end{equation}
and that  $v^i_j$ is an equi-bounded sequence in $W^{2,2}(\R;X)\cap L^\infty(\R;X)$, with the bounds independent of $i$ and $j$, provided by \eqref{equa1n}-\eqref{equa3} along the sequence $u_k$. By the Ascoli-Arzel\`a Theorem, there exists a function $v^i$ such that, up to a (not relabeled) subsequence, $v^i_j\to v^i$, $\dot{v}^i_j\to\dot{v}^i$ uniformly on the compact subsets of $\R$, and $v^i_j\to v^i$ weakly in $W^{2,2}(\R;X)$.
By semicontinuity, from \eqref{condF} we obtain, for every arbitrary closed interval $[a,b]\subset\R$,
\begin{equation*}
\frac{1}{2}\int_{a}^{b}\|A\ddot{v}^i(r)+B\dot{v}^i(r)+\nabla_xF_t(v^i(r))\|_{B^{-1}}^2\,\mathrm{d}r\leq0
\end{equation*}
whence we finally get
\begin{equation*}
A\ddot{v}^i(r)+B\dot{v}^i(r)+\nabla_xF_t(v^i(r))=0
\end{equation*} 
for every $r \in \mathbb{R}$.
As in Step 1, 
we have $v^i\in W^{2,2}(\R;X)\cap L^\infty(\R;X)$ and
\begin{equation*}
\int_{-\infty}^{+\infty}\|\nabla_xF_t(v^i(r))\|^2\,\mathrm{d}r\leq M.
\end{equation*}
Therefore, again by virtue of Lemma~\ref{Lemma1}, we get the existence of the limits $\displaystyle\lim_{r\to-\infty}v^i(r)=v^i_*\in\mathcal{C}_M(t)$, $\displaystyle\lim_{r\to+\infty}v^i(r)=v^{i,*}\in\mathcal{C}_M(t)$ and $\displaystyle\lim_{r\to\pm\infty}\dot{v}^i(r)=0$. Passing to the limit as $j\to+\infty$ in \eqref{boundbis} we deduce that
\begin{equation}
\|v^i(\tau)-u^{i-1}\|\leq d\le \frac14d_t
\label{boundtris}
\end{equation}
for all $\tau \le 0$. Since $v^i_*\in\mathcal{C}_M(t)$, by our choice of $d_t$ we conclude that it must be $v^i_*=u^{i-1}$. 

Now we show that $v^{i,*}\neq u^l$, for every $l\leq i-1$. In order to do that, we first note that for every $s\geq0$ it holds
\begin{equation}
\begin{split}
F_t(v^i(s))+\frac{1}{2}\|\dot{v}^i(s)\|_A^2&\leq \displaystyle\mathop{\lim\inf}_{j\to+\infty}\Bigl(F_t(v^{i-1}_j(t_j))+\frac{1}{2}\|\dot{v}^{i-1}_j(t_j)\|_A^2\\
&+\int_{-(t^i_j-t_j)}^s\langle \nabla_xF_t(v^i_j(r))+A\ddot{v}^i_j(r),\dot{v}^i_j(r)\rangle\,\mathrm{d}r\Bigr).
\end{split}
\label{condizz}
\end{equation}
By \eqref{condC}, for $s=0$, we have that $F_t(v^{i-1}_j(t_j))\to F_t(u^{i-1})$ and $\|\dot{v}^{i-1}_j(t_j)\|_A^2\to0$. Furthermore, with \eqref{condF} we also get
\begin{equation}\label{above}
\begin{split}
&\displaystyle\mathop{\lim\inf}_{j\to+\infty}\frac{1}{2}\int_{-(t^i_j-t_j)}^s\|A\ddot{v}^i_j(r)+B\dot{v}^i_j(r)+\nabla_xF_t(v^i_j(r))\|_{B^{-1}}^2\,\mathrm{d}r\\
&\leq \displaystyle\mathop{\lim}_{j\to+\infty}\frac{1}{2}\int_{-\infty}^{+\infty}\|A\ddot{v}^i_j(r)+B\dot{v}^i_j(r)+\nabla_xF_t(v^i_j(r))\|_{B^{-1}}^2\,\mathrm{d}r=0.
\end{split}
\end{equation}
Observe that, by \eqref{simpleid} with $Q=B$ we have for  every $r$
\[
\begin{split}
&\displaystyle2\langle \nabla_xF_t(v^i_j(r))+A\ddot{v}^i_j(r),\dot{v}^i_j(r)\rangle=\\
\displaystyle\|A\ddot{v}^i_j(r)+B\dot{v}^i_j(r)+\nabla_xF_t(v^i_j(r))\|_{B^{-1}}^2&-\Big\{\|\nabla_xF_t(v^i_j(r))+A\ddot{v}^i_j(r)\|_{B^{-1}}^2+\|\dot{v}^i_j(r)\|_B^2\Big\}\,.
\end{split}
\]
Combining this with \eqref{above}, it holds for all $s \ge 0$
\begin{equation}\label{fin?}
\begin{split}
&\displaystyle\mathop{\lim\inf}_{j\to+\infty}\int_{-(t^i_j-t_j)}^s\langle \nabla_xF_t(v^i_j(r))+A\ddot{v}^i_j(r),\dot{v}^i_j(r)\rangle\,\mathrm{d}r=\\
&\leq \displaystyle\mathop{\lim\inf}_{j\to+\infty}\Big(-\frac{1}{2}\int_{-(t^i_j-t_j)}^s\|\nabla_xF_t(v^i_j(r))+A\ddot{v}^i_j(r)\|_{B^{-1}}^2+\|\dot{v}^i_j(r)\|_B^2\,\mathrm{d}r\Big)\\
&\leq \displaystyle\mathop{\lim\inf}_{j\to+\infty}\Big(-\frac{1}{2}\int_{-(t^i_j-t_j)}^0\|\nabla_xF_t(v^i_j(r))+A\ddot{v}^i_j(r)\|_{B^{-1}}^2+\|\dot{v}^i_j(r)\|_B^2\,\mathrm{d}r\Big)\,,
\end{split}
\end{equation}
where the last inequality is simply due to assuming $s \ge 0$.
Since $v^i_j(-(t^i_j-t_j))=v^{i-1}_j(t_j)\to u^{i-1}$, while $\|v^i_j(0)-u^{i-1}\|=d$,  by arguing as for \eqref{conddelta}, there exists $\gamma_i=\gamma_i(t,d,u^{i-1},M,\beta)>0$ such that
\begin{equation}
-\frac{1}{2}\int_{-(t^i_j-t_j)}^0\|\nabla_xF_t(v^i_j(r))+A\ddot{v}^i_j(r)\|_{B^{-1}}^2+\|\dot{v}^i_j(r)\|_B^2\,\mathrm{d}r\leq -\gamma_i,
\label{condizz1}
\end{equation}
for every fixed $j$.
With \eqref{condizz}, \eqref{fin?},  and \eqref{condizz1} we then conclude that
\begin{equation*}
\begin{split}
F_t(v^i(s))+\frac{1}{2}\|\dot{v}^i(s)\|_A^2&\leq F_t(u^{i-1})-\gamma_i\\
&\leq F_t(u^{i-2})-\gamma_i-\gamma_{i-1}\\
&\ldots\\
&\leq F_t(u^0)-\sum_{k=0}^{i-1}\gamma_i,
\end{split}
\end{equation*}
for all $s \ge 0$, whence, passing to the limit as $s\to+\infty$ in each of the previous inequalities, we get
\begin{equation*}
F_t(v^{i,*})<F_t(u^l),\quad\text{ for every }l=0,\dots,i-1.
\end{equation*}
Thus, $v^{i,*}\neq u^{l}$ for every $l=0,\dots,i-1$. Setting $u^i:=v^{i,*}$, if it results $u^i=v$ then $m=i$ and the proof stops here, otherwise the construction goes on and it will stop after a finite number of steps since $\mathcal{C}_M(t)$ is finite and $u_k(+\infty)=v$.

\endproof

\begin{oss}\label{fine}
If a jump occurs at $t=0$ and $u_-(0)\notin \mathcal{C}(0)$, the statement of the previous Theorem has to be modified as follows. One sets $u^0=u_-(0)$ and the point $u^1$ is obtained as limit for $t\to +\infty$ of a solution $v^1$ to the problem
\[
\begin{cases}
&A\ddot{v}^1(s)+B\dot{v}^1(s)+\nabla_xF_t(v^1(s))=0,\quad\forall\,s\in\R;\\
&v^1(-1)=u^{0}, \quad \dot{v}^1(-1)=0.
\end{cases}
\] 
Notice that above the initial condition can be given at any time other than $t=-1$, since solutions are invariant for time-translations. If $m>1$, then the procedure goes on exactly as in \eqref{properties}.
\end{oss}

\section{Appendix}

We discuss here a  simple explicit example of a (discontinuous) quasistatic evolution originating as limit of \eqref{mainequation}. By doing this, we also highlight the difference between the limit evolutions in our case, compared with the singular limit of gradient flows considered in \cite{Ago-Rossi, Sci-Sol, Zanini}.  Indeed, we will show that the two evolutions, coinciding until the first jump time $t^*=1$, then have a different behavior in the right neighborhood of $t^*$: while the rescaled gradient flow is actually stuck to the closest potential well, a second-order dynamics can overcome an energy barrier, reaching a different limit point.

For this, we consider a non-convex driving energy $F:I\times\R\longrightarrow\R$, with $I:=[0,+\infty)$, defined as
\begin{equation}
F_t(x):=-x(t-1+x^2)+P(x),
\label{driven}
\end{equation}
where $P(x)$ is a function such that $P(x)=0$ if $x\leq0$, $P(x)/x^3\to0$ as $x\to0^+$ (in order to have $F_t(x)\in C^3(I\times\R)$) and $P(x)/x^3\to+\infty$ as $x\to+\infty$ (in order to get the necessary equicoerciveness).

It is easy to see that the function $\varphi(t):=-\sqrt{\frac{1-t}{3}}$, for $t\in[0,1)$, is a curve of local minimizers of $F_t(x)$. Now, for $\epsilon>0$ let $u_{1,\epsilon}$ and $u_{2 ,\epsilon}$ be the solutions of the problem
\begin{equation*}
\begin{cases}
\frac{\epsilon}{4}\dot{u}_{1,\epsilon}(t)=-F_t'(u_{1,\epsilon}(t)),\quad t\in I\\
{u}_{1,\epsilon}(0)=-(1+\epsilon)\sqrt{\frac13}
\end{cases}
\end{equation*}
and
\begin{equation*}
\begin{cases}
\epsilon^2\ddot{u}_{2,\epsilon}(t)+\frac{\epsilon}{4}\dot{u}_{2,\epsilon}(t)=-F_t'(u_{2,\epsilon}(t)), \quad t\in I\\
{u}_{2,\epsilon}(0)=-(1+\epsilon)\sqrt{\frac13}\\
\dot{u}_{2,\epsilon}(0)=1\,,
\end{cases}
\end{equation*}
respectively. It follows from \cite[Lemma 3.1]{Zanini} and \cite[Proposition 3.4 and Remark 3.9]{Ago1}, respectively, that the limit evolutions $u_1, u_2$ satisfy
\begin{equation*}
u_1(t)=u_2(t)=\varphi(t),\quad \mbox{ for every }t\in[0,1).
\end{equation*}
In particular, it holds that
\begin{equation*}
u_1(1-)=u_2(1-)=0,
\end{equation*}
while for $t=1$ we must have a jump, since $\varphi(t)$ is not defined for $t>1$. Our goal is now to show that, up to suitable choosing $P(x)$ in the definition of the driving energy \eqref{driven}, we have
$u_1(1+)\neq u_2(1+)$.
More precisely, let $P(x)$ be such that $F_1(x):=F(1,x)$ complies with the following assumptions:\\
(H1) $F_1'(x)=0$ if and only if $x\in\{0, 1, 2, 9\}$;\\
(H2) $F_1(1)=-1$;\\
(H3) $F_1(2)=-1+\eta$, $\eta>0$;\\
(H4) $F_1(3)=-3$ and $F_0'(x)\leq-1$ for every $x\in[3,8]$;\\
(H5) $F_1''(1)>0$ and $F_1''(9)>0$.
\\
Under the above assumptions, we claim that
\begin{equation}
u_1(1+)=1\neq 9=u_2(1+).
\label{assertion}
\end{equation}

To prove the claim, observe that by assumption (H5), the points $x=1$ and $x=9$ are strict local minimizers of $F_1$. Hence, no heteroclinic orbit of
$\frac14\dot v=-F'_1(v)$ and of $\ddot v+\frac14\dot v=-F'_1(v)$ can start from  $v=1$ or $v=9$ for $s\to -\infty$. With this, \eqref{assertion} is  a direct consequence of \cite[Proposition 1]{Zanini} (for $u_1$) and Theorem \ref{variat} (for $u_2$), once we prove that the heteroclinic solutions of $\frac14\dot v=-F'_1(v)$, and of $\ddot v+\frac14\dot v=-F'_1(v)$ starting from  $v=0$ for $s\to -\infty$ have $v=1$, and $v=9$, respectively, as a limit when $s\to +\infty$. 
We are therefore only left to prove the following Proposition. Below, we set $F=F_1$ for brevity.

\begin{prop}\label{lemmone}
Assume that $F$ complies with assumptions (H1)-(H5) above. Let $v_1$ be the unique (up to translations in time) solution of the first order problem
\begin{equation}
\begin{cases}
\frac{1}{4}\dot{v}=-F'(v),\\
v(-\infty)=0,
\end{cases}
\label{problema1}
\end{equation}
and $v_2$ be the unique (up to translations in time) solution of the second order problem
\begin{equation}
\begin{cases}
\ddot{v}+\frac{1}{4}\dot{v}=-F'(v),\\
v(-\infty)=0,\\
\dot{v}(-\infty)=0.
\end{cases}
\label{problema2}
\end{equation}
Then we have that
\begin{equation}
\lim_{s\to+\infty}v_1(s)=1,
\label{cond1}
\end{equation}
and
\begin{equation}
\lim_{s\to+\infty}v_2(s)=9.
\label{cond2}
\end{equation}
\end{prop}

\proof
The existence and uniqueness (up to translations in time) of $v_1$ and $v_2$ are a consequence of \cite[Lemma 2.1]{Zanini} and \cite[Lemma 2.5]{Ago1}, respectively. 
Since the total energy is nonincreasing along the evolutions, and $F_1(v)>0$ for $v<0$, we have $v_1(t), v_2(t)>0$ for all $t$. Since $v=1$ is a local minimum of the energy, \eqref{cond1} is an immediate consequence of the gradient flows' properties. In order to prove \eqref{cond2},  we observe that the $\omega$-limit set of $(v_2(s), \dot v_2(s))$ can only consist of equilibria, since we have a second-order scalar dissipative equation.  Hence, it will be sufficient to show that
\begin{equation}
\lim_{s\to+\infty}v_2(s)>2,
\label{tesi}
\end{equation}
since $v=9$ is then the unique available limit point.

We preliminarly  note that we can fix $s_0\in\R$ such that $\dot{v}_2(s_0)>0$, $v_2(s_0)\in[0,\eta]$ and $\int_{-\infty}^{s_0}|\dot{v}_2(s)|^2\,\mathrm{d}s\leq\eta$. We argue by contradiction, and assume that $v_2(+\infty)\leq 2$. In this case, it will be either $v_2(+\infty)=1$ or $v_2(+\infty)=2$. Setting
\begin{equation*}
D:=\frac{1}{4}\int_\R|\dot{v}_2(s)|^2\,\mathrm{d}s,
\end{equation*}
in the first case by (H2) we have $D=1$, while in the second one it results $D=1-\eta$ by (H3). Thus,
\begin{equation}
D\leq1.
\label{(**)}
\end{equation}

{\bf Claim 1:} there exists $t_0\in\R$ such that $v_2(t_0)=1$ and $\dot{v}_2(t_0)>0$.\\
For this, we assume by contradiction that $v_2(t)<1$ for every $t\in[s_0,+\infty)$. In this case, we have $v_2(+\infty)=1$. Furthermore, it holds $\dot{v}_2(t)>0$ for every $t$. Indeed, if not, since $\dot{v}_2(s_0)>0$, it would exist $t$ such that $\dot{v}_2(t)=0$ and $\ddot{v}_2(t)\leq0$. But this contradicts the equation $\ddot{v}_2(t)=-\frac{1}{4}\dot{v}_2(t)-F'(v_2(t))$, since $F'(v)<0$ for $v\in(0,1)$. Now, one the one hand from $v_2(+\infty)=1$ we infer that
\begin{equation}
\frac{1}{4}\int_\R|\dot{v}_2(s)|^2\,\mathrm{d}s=F(v_2(-\infty))-F(v_2(+\infty))=F(0)-F(1)=1.
\label{stimone0}
\end{equation}
On the other hand,
\begin{equation}
\begin{split}
\frac{1}{4}\int_\R|\dot{v}_2(s)|^2\,\mathrm{d}s&=\frac{1}{4}\int_{-\infty}^{s_0}|\dot{v}_2(s)|^2\,\mathrm{d}s+\frac{1}{4}\int_{s_0}^{+\infty}|\dot{v}_2(s)|^2\,\mathrm{d}s\\
&\leq \frac{\eta}{4}+\frac{1}{4}\int_{s_0}^{+\infty}\dot{v}_2(s)\cdot\dot{v}_2(s)\,\mathrm{d}s.
\end{split}
\label{stimone1}
\end{equation}
Since
\begin{equation*}
\int_{s_0}^{+\infty}\dot{v}_2(s)\,\mathrm{d}s=1-v_2(s_0)\leq1
\end{equation*}
and
\begin{equation*}
\frac{1}{2}|\dot{v}_2(t)|^2\leq F(0)-F(v_2(t))\leq F(0)-F(1)=1,\quad\mbox{ for every }t\in\R,
\end{equation*}
by \eqref{stimone1} we can deduce that
\begin{equation*}
\frac{1}{4}\int_\R|\dot{v}_2(s)|^2\,\mathrm{d}s\leq\frac{\eta+\sqrt{2}}{4},
\end{equation*}
which is in contradiction with \eqref{stimone0}. Then, there exists $t_0$ such that $v_2(t_0)=1$. As remarked before, $\dot{v}_2(t_0)\geq0$, but since $v=1$ is a critical point of the energy, it must be $\dot{v}_2(t_0)>0$.\\
{\bf Claim 2:} there exists $t_1\in\R$ such that $v_2(t_1)=2$ and $\dot{v}_2(t_1)>0$.\\
As before, we argue by contradiction and assume that $v_2(t)\le 2$ for every $t\in[t_0,+\infty)$. 
Set $c_\eta=\frac12\min\{\dot v_2(t_0), \eta\}$ and let $\hat t$ be the smallest time in $(t_0, +\infty)$ with $\dot v_2(\hat t)=c_\eta$, which must exist since $\dot v_2(t)\to 0$ when $t\to +\infty$. Notice that, since $\dot{v}_2(t)>0$ for every $t\in[s_0, t_0]$, we also have $\dot{v}_2(t)>0$ for every $t\in[s_0, \hat t]$. Moreover, we get
\begin{equation}
\frac12 c_\eta^2+\frac{1}{4}\int_{-\infty}^{\hat{t}}|\dot{v}_2(s)|^2\,\mathrm{d}s=F(0)-F(v_2(\hat{t}))\geq 1-\eta,
\label{stimone3}
\end{equation}
since $1= v_2(t_0)\le  v_2(\hat{t})\le 2$.
Furthermore, for every $t\in(-\infty,\hat{t}]$ it holds that
\begin{equation*}
\frac{1}{2}|\dot{v}_2(t)|^2\leq F(0)-F(v_2(t))\leq F(0)-F(1)=1\,.
\end{equation*}
We then obtain
\begin{equation*}
\begin{split}
\frac{1}{4}\int_{-\infty}^{\hat{t}}|\dot{v}_2(s)|^2\,\mathrm{d}s&\leq \frac{\eta}{4}+\frac{1}{4}\int_{s_0}^{\hat{t}}|\dot{v}_2(s)|^2\,\mathrm{d}s\leq\frac{\eta}{4}+\frac{\sqrt{2}}{4}\int_{s_0}^{\hat{t}}\dot{v}_2(s)\,\mathrm{d}s\\
&\leq \frac{\eta+2\sqrt{2}}{4},
\end{split}
\end{equation*}
which contradicts \eqref{stimone3} for $\eta$ small. Then, the smallest time $t_1$ where $v_2(t_1)=2$ is well defined and, as $v=2$ is an equilibrium, it must be $\dot{v}_2(t_1)>0$.

Now, since $F'(v)<0$ for every $v\in[2,9]$, as already done for Claim 1 we can show that, if $M>0$ is chosen such that $v_2(t)<9$ for every $t\in[t_1,t_1+M]$, then $\dot{v}_2(t)>0$ in $[t_1,t_1+M]$. As a consequence, there exist $t_2\leq t_3$ such that 
\begin{equation}\label{inter}
v_2(t_2)=3\leq v_2(t) \leq 8=v_2(t_3),\quad \forall t\in[t_2,t_3].
\end{equation}
In addition,
\begin{equation*}
\frac{1}{2}|\dot{v}_2(t_2)|^2+\frac{1}{4}\int_{-\infty}^{t_2}|\dot{v}_2(r)|^2\,\mathrm{d}r=F(0)-F(3)=3.
\end{equation*}
This implies, with \eqref{(**)}, that
\begin{equation*}
3\leq \frac{1}{2}|\dot{v}_2(t_2)|^2+D\leq\frac{1}{2}|\dot{v}_2(t_2)|^2+1,
\end{equation*}
from which $\dot{v}_2(t_2)\geq2$. With this, using the equation \eqref{problema2},  \eqref{inter}, and assumption (H4), we can show with similar arguments as before that $\dot{v}_2(t)>1$ for every $t\in[t_2,t_3]$. From this, we finally deduce
\begin{equation*}
\begin{split}
D&\geq \frac{1}{4}\int_{t_2}^{t_3}|\dot{v}_2(r)|^2\,\mathrm{d}r=\frac{1}{4}\int_{t_2}^{t_3}\dot{v}_2(r)\cdot\dot{v}_2(r)\,\mathrm{d}r\geq \frac{1}{4}\int_{t_2}^{t_3}\dot{v}_2(r)\,\mathrm{d}r\\
&=\frac{1}{4}(v_2(t_3)-v_2(t_2))=\frac{8-3}{4}>1,
\end{split}
\end{equation*}
thus contradicting \eqref{(**)}. Therefore, \eqref{tesi} must hold.
\endproof



\end{document}